\documentclass[11pt]{amsart}
\usepackage{amsmath, palatino, mathpazo, amsfonts, amssymb, mathtools}
\usepackage[mathscr]{eucal}
\usepackage[all]{xy}
\usepackage{datetime}

\usepackage{amsthm}
\usepackage {tikz}
\usepackage {tgbonum}                 
\usepackage {amsmath}                   % for AMS math environments
\usepackage {amssymb}                   % for AMS math symbol definitions
\usepackage {bibnames}                  % for BibTeX- and TeX-related names
\usepackage {cite}                      % for numbered and ranged numeric citations
\usepackage {url}                       % for better typesetting of URLs and pathnames
\usepackage {varioref}                  % for improved document cross-references

\usepackage{tabularx}
\usepackage{array}

\newtheorem{theorem}{Theorem}[section]
\newtheorem{lemma}[theorem]{Lemma}
\newtheorem{proposition}[theorem]{Proposition}
\newtheorem{corollary}[theorem]{Corollary}
\newtheorem{conjecture}[theorem]{Conjecture}

\theoremstyle{remark}
\newtheorem{remark}[theorem]{Remark}
\newtheorem{example}[theorem]{Example}
\newtheorem{definition}[theorem]{Definition}
\newtheorem{conventions}[theorem]{Conventions}

\numberwithin{equation}{section}

\newcommand{\p}{\partial}

\newcommand{\op}{\operatorname}
\newcommand{\PP}{\mathbb{P}}

\newcommand{\M}{\overline{\mathcal{M}}}
\newcommand{\K}{\mathcal{K}}
\newcommand{\X}{\mathcal{X}}

\newcommand{\vir}{{\rm vir}}
\newcommand{\ev}{{\rm ev}}
\newcommand{\Res}{{\rm Res}}
\newcommand{\td}{{\rm td}}
\newcommand{\ch}{{\rm ch}}
\newcommand{\tr}{{\rm tr}}
\newcommand{\Tr}{{\rm Tr}}
\newcommand{\Aut}{{\rm Aut}}
\newcommand{\fake}{{\rm fake}}

\newcommand{\Ext}{{\rm Ext}}
\newcommand{\KK}{{K}}
\newcommand{\Coeff}{{\rm Coeff}}
\newcommand{\leg}{{\rm leg}}

\newcommand{\tail}{{\rm tail}}

\def\O{\mathcal{O}}

\newcommand{\GW}{{\mathrm{GW}}}
\newcommand{\GV}{{\mathrm{GV}}}
\newcommand{\dd}{{\mathbf{d}}}
\newcommand{\e}{\epsilon}
\newcommand{\tw}{\mathrm{tw}}

\title[QK $=$ GV I. The quintic $3$-fold]{Quantum $K$-invariants and Gopakumar--Vafa invariants I. \\ The quintic threefold}
	
	\author{Y.-C.~Chou}
	\email{bensonchou72@gmail.com, chou@math.utah.edu}
	
	\author{Y.-P.~Lee}
	\email{yplee.math@gmail.com, ypleemath@gate.sinica.edu.tw}
	
\address{Institute of Mathematics, Academia Sinica, Taipei 10617, Taiwan, and
Department of Mathematics, University of Utah, 	Salt Lake City, Utah 84112-0090, U.S.A.}

\date{\today}
\begin{document}

\maketitle

\begin{abstract}
We prove a conjecture of Jockers--Mayr and Garoufalidis--Scheidegger, relating genus zero quantum $K$-invariants and Gopakumar--Vafa invariants on the quintic threefold.
\end{abstract}

\tableofcontents

\setcounter{section}{-1}

\section{Introduction}

\subsection{Relating enumerative invariants on quintic threefold}
On the Calabi--Yau threefolds (CY3), there are two sets of \emph{integral} enumerative invariants.
The first, called the \emph{Gopakumar--Vafa invariants} \cite{Gopakumar_Vafa_1, Gopakumar_Vafa_2}, was introduced 
in theoretical physics as ``new topological invariants on \emph{Calabi--Yau threefolds}'', counting the
``numbers of BPS states''.  We refer the readers to \cite{Maulik_Toda_2018} and references therein for various mathematical definitions of these invariants. There is a simple relation between the Gopakumar--Vafa invariants and the Gromov--Witten invariants \eqref{e:1.1} \eqref{e:1.2}, 
\begin{equation*} %\label{e:1.1}
\begin{split}
 &\sum _{g=0}^{\infty }~\sum _{\beta \in H_{2}(M,\mathbb {Z} )}{\GW}_{g,\beta} q^{\beta }\lambda ^{2g-2} \\
 = & \sum _{g=0}^{\infty }~\sum _{k=1}^{\infty }~\sum _{\beta \in H_{2}(M,\mathbb {Z} )}{\GV}_{g,\beta}{\frac {1}{k}}\left(2\sin \left({\frac {k\lambda }{2}}\right)\right)^{2g-2}q^{k\beta } ,
\end{split}
\end{equation*}
which all viable mathematical definitions must satisfy. In this paper, we will use the above (invertible) relation as the definition of the Gopakumar--Vafa invariants (in terms of Gromov--Witten invariants). The integrality of this ad hoc definition was proven by E.~Ionel and T.~Parker \cite{Ionel_Parker_2018}.

The second is the \emph{quantum $K$-invariants} \cite{Givental_2000, Lee_2004}, a $K$-theoretic variant of the (cohomological) Gromov--Witten invariants. Whereas Gromov--Witten theory produces \emph{rational} enumerative invariants, quantum $K$-theory, counting (the alternating sum of) the rank of sheaf cohomology, produces \emph{integral} invariants by definition.
The close relationship between the quantum $K$-theory and quantum cohomology was understood since the early phase of quantum $K$-theory. In fact, many $K$-theoretic results were inspired by their cohomological counterparts. See, e.g., \cite{Givental_Lee_2003, Lee_2004}. In %a groundbreaking work 
\cite{Givental_Tonita_2011} and subsequent works, A.~Givental and his collaborators completely characterized genus zero quantum $K$-invariants in terms of genus zero Gromov--Witten invariants. 

Therefore, there is a relation between quantum $K$-invariants and Gopakumar--Vafa invariants for Calabi--Yau threefolds via Gromov--Witten invariants. In \cite{Jockers_Mayr_2019} and  \cite[Conjecture~1.1]{Garoufalidis_Scheidegger_2022}, 
H.~Jockers, P.~Mayr and S.~Garoufalidis, E.~Scheidegger proposed a simple \emph{linear} relation between these two sets of integral invariants for \emph{quintic threefold} at genus zero. They also provided numerical evidence for this conjecture. 
The main purpose of this article is to prove this conjecture as stated in \cite[Conjecture~1.1]{Garoufalidis_Scheidegger_2022}, cf.\ Conjecture~\ref{conjectureJK} in Section~\ref{s:1}. According to \cite[Section 3.3]{Garoufalidis_Scheidegger_2022}, this also proves the version proposed by Jockers and Mayr.

We plan to generalize these results to other Calabi--Yau threefolds (Part II) and to higher genus (Part III) in future works.
In particular, some heuristic relations to the ``multiple cover formula'' will be discussed in Part II, where the proof of genus $0$ result for all Calabi--Yau threefolds will use different arguments.

%There have been various attempts at defining Gopakumar--Vafa invariants mathematically. We refer the readers to \cite{Maulik_Toda_2018} and references therein.

%Unlike the Gromov--Witten invariants, which are defined for any symplectic manifolds or orbifolds, Gopakumar--Vafa invariants only make sense for Calabi--Yau threefolds (or variants).
\subsection{Contents of the paper}
Conjecture~\ref{conjectureJK} is phrased in terms of an explicit expression for the \emph{small $J^K$-function} in quantum $K$-theory, which is a generating function of $K$-theoretic Gromov-Witten invariants, or quantum $K$-invariants. For a smooth projective variety $X$, the quantum $K$-invariants are defined as
\[
\langle \tau_{d_1}(\Phi_1)\dots \tau_{d_n}(\Phi_n) \rangle^{X, {K}}_{g,n,\beta} :=  \chi \Big(  \M_{g,n}(X,\beta) ; \Big(  \otimes_{i=1}^n \ev_i^*(\Phi_i) L_i^{d_i}  \Big) \otimes \O^{\vir}  \Big) \in \mathbb{Z}.
\]
It gives a deformation of the ordinary $K$-ring $K^0(X)$ of $X$, analogous to the relation between quantum cohomology and ordinary cohomology. We consider the generating function, the \emph{big $J^K$-function}, which determines the \emph{genus $0$} quantum $K$-theory of $X$:
\[
\begin{split}
       t \rightarrow & J^{\KK}(t, q, Q) \\
       & := (1-q) + t(q) + \sum_{\alpha} \Phi_{\alpha}\sum_{n,d} \frac{Q^d}{n!} \langle \frac{\Phi^{\alpha}}{1-qL}, t(L),\dots,t(L) \rangle ^{X,\KK}_{0,n+1,d}.
\end{split}
\]
The \emph{small} $J$-function is a specialization of the big $J$-function to $t=0$. Jockers–Mayr and Garoufalidis–Scheidegger formulated an explicit expression of the small $J^K$-function in terms of the Gopakumar--Vafa invariants.
We recall the basic definitions and formulations of quantum $K$-theory and Gopakumar--Vafa invariants in Section~\ref{s:1}. The Conjecture~\ref{conjectureJK} is formulated in Section~\ref{section_conjecture}.

In Section~\ref{s:2}, we recall necessary formulations and results from the works of A.~Givental and his collaborators. There are generally two approaches to compute the genus zero quantum $K$-theory for the quintic threefold.
In \cite{Givental_Tonita_2011}, Givental and Tonita prove that genus 0 $K$-theoretic GW-invariants of $X$ can be expressed in terms of its cohomological ones. Their framework can be understood as the \emph{adelic characterization} of quantum $K$-theory. This is done by applying a virtual version of Kawasaki's Hirzebruch--Riemann--Roch formula, expressing Euler characteristics of vector bundles over a Deligne--Mumford \emph{stack} (moduli of stable maps in our case) in terms of integrations over the components of its inertia stack, called Kawasaki strata. 
%We recall the Kawasaki's formula in Section \ref{section_KHRR}. 
By analyzing the combinatorial structure of inertia stacks of moduli of stable map, they derive a recursive relation relating invariants in different strata and then give the characterization of genus $0$ $K$-theoretic GW-invariants in terms of the cohomological ones. See Section \ref{section_KHRR_recursion} for details.

%In Section \ref{section_conjecture}, we state the Conjecture \ref{conjectureJK} in \cite{Garoufalidis_Scheidegger_2022} and introduce their approach. 
Another approach is called the \emph{explicit reconstruction theorem}, which in principal gives an algorithm to compute all quantum $K$-invariants from an initial condition. There are various versions of the reconstruction theorem.  A particular case of sufficient initial condition for the quintic threefold is given by Givental's $I^K$-function. See \cite{Lee_Pandharipande_2004},  \cite{Givental_PEV_Toric_2015} and \cite{Tonita_TwistedK_2018}. 

In principal, either the \emph{virtual orbifold Hirzebruch--Riemann--Roch} or the \emph{reconstruction theorem} gives an explicit algorithm to compute the $J^K$-function. %with any input. 
In particular, a computation up to degree $7$ is given in \cite{Garoufalidis_Scheidegger_2022} using the explicit reconstruction alone.
However, the complexity for explicit computation grows very fast as the degrees increase. %, and we are not able to find good solutions using only one method.
%the pattern and derive a general formula. 
In Section~\ref{section_proof}, we combine both methods and utilize the \emph{analytic properties} of the generating functions to give a proof of Conjecture~\ref{conjectureJK}. 

To streamline the proof, some computations needed to establish the proof are left to Section~\ref{appen_Kawasaki}. In the appendix, an alternative proof of Theorem~\ref{Thm_deg0_inv} is given via the explicit reconstruction. This is to show that one can in principle prove a theorem using either reconstruction theorem or Riemann--Roch, when the complexity of the former is manageable.

\subsection*{Acknowledgements}
We wish to thank A.~Givental, R.~Pandharipande, E.~Scheidegger and H.-H.~Tseng for their interest and discussions. The results was presented in Bumsig Kim memorial conference (KIAS) and ETH (Z\"urich) in September and October 2022.
This research is partially supported by the Simons Foundation, the NSTC, Academia Sinica and University of Utah.

\section{GW, GV, QK and the JMGS Conjecture} \label{s:1}
%\section{Gopakumar--Vafa invariants and Gromov--Witten invariants}

\subsection{Gromov--Witten invariants}
Let $X$ be a smooth complex projective variety, and $\M_{g,n}(X,\beta)$ be the M.~Kontsevich’s moduli space of $n$-pointed, genus $g$, degree $\beta$ stable maps. 
Given $i\in \{ 1,\dots,n\}$, there is an evaluation map 
\[
\begin{split}
\ev_i : \M_{g,n}(X,\beta) & \rightarrow X
\\
[f: (C;x_1,\dots,x_n)] & \mapsto f(x_i),
\end{split}
\]
and a line bundle $L_i := x_i^* w_{\mathcal{C}/\M}$ on $\M_{g,n}(X,\beta)$, where $w_{\mathcal{C}/\M}$ is the relative dualizing sheaf of the universal curve $\mathcal{C} \rightarrow \M_{g,n}(X,\beta)$ and $x_i: \M_{g,n}(X,\beta) \rightarrow \mathcal{C}$ is the $i$-th mark point. 

(Cohomological) Gromov-Witten invariants of $X$ are defined to be
\[
\langle \tau_{d_1}(\phi_1)\dots \tau_{d_n}(\phi_n) \rangle^{X, {H}}_{g,n,\beta} :=  \pi^H_* \left( \cup_{i=1}^n \ev_i^*(\phi_i) c_1(L_i)^{d_i} \cap [\M_{g,n}(X,\beta)]^{\vir} \right) \in \mathbb{Q},
\]
where 
\[
 \pi: \M_{g,n}(X,\beta) \to pt := \op{Spec}(\mathbb{C})
\]
is the structural map and $\pi^H_*$ is the (cohomological) pushforward to the point.
Here $\phi_1,\dots,\phi_n \in H(X)$, $d_1,\dots,d_n \in \mathbb{Z}_{\geq 0}$, and $[\M_{g,n}(X,\beta)]^{\vir}$ are the (cohomological) virtual fundamental classes. 

In this paper, we will be concerned with only genus zero invariants, which can be encoded into a formal power series, called genus-0 descendant potential of $X$:
\[
F^{{H}}_0(t) = \sum_{n\geq 0}\sum_{\beta} \frac{Q^{\beta}}{n!} \langle  t(L),\dots,t(L)  \rangle^{X,{H}}_{0,n,\beta}.
\]
Here the sum is over all curve class $\beta \in H_2(X)_{\geq 0}$ and $Q^{\beta}$ are formal variables, called the Novikov variables, which keep track of the curve classes. $t(q)$ stands for any polynomial of one variable with coefficients in $H(X)$.
That is $\displaystyle t(q) = \sum_{k\in \mathbb{Z}_{\ge 0}} \sum_{\alpha=1}^N t_{k}^{\alpha} \phi_{\alpha} q^k$ with $\{\phi_{\alpha} \}_{\alpha=1}^N$ a basis in $H(X)$.

\subsection{Gopakumar--Vafa invariants}
In theoretical physics, R.~Gopakumar and C.~Vafa in \cite{Gopakumar_Vafa_1, Gopakumar_Vafa_2} introduced new topological invariants on \emph{Calabi--Yau threefolds} (CY3) $X$, called \emph{Gopakumar--Vafa invariants}.
These invariants represent the counts of ``numbers of BPS states'' on $X$. Unlike the Gromov--Witten invariants, which are defined for any symplectic manifolds or orbifolds, Gopakumar--Vafa invariants only make sense for Calabi--Yau threefolds (or variants).

The virtual dimensions for moduli spaces (of stable maps) to CY3 are always equal to the number of marked points. That is,
\[
 \op{vdim} \M_{g,n} (X, \beta) = n.
\]
By general properties of the moduli spaces, more precisely the string equation, the divisor equation and dilaton equation, all Gromov--Witten invariants can be easily reconstructed from
\[
 {\GW}_{g, \beta} := \pi^H_* \left( [\M_{g,0}(X,\beta)]^{\vir} \right) = \int_{[\M_{g,0}(X,\beta)]^{\vir} } 1 .
\]
That is, all $n$-pointed invariants with descendants can be recovered from $0$-pointed counting.
In fact, a \emph{closed formula} was obtained in \cite{Fan_Lee_2019} in terms of generating functions. We will therefore focus on ${\GW}_{g, \beta}$.

%Parallel story works for the Gopakumar--Vafa invariants. 
There have been various attempts at defining Gopakumar--Vafa invariants mathematically. We refer the readers to \cite{Maulik_Toda_2018} and references therein. Currently, the Gopakumar--Vafa invariants are defined for $0$-pointed curves only, without insertions. Generalizations to $n$-pointed invariants with insertion are expected to be compatible with the dilaton and divisor equations. In that case, the counting of the BPS states is again reduced to similarly defined ${\GV}_{g, \beta}$.

A remarkable relation stated in \cite{Gopakumar_Vafa_1, Gopakumar_Vafa_2} between GV and GW can be expressed in terms of generating functions:
\begin{equation} \label{e:1.1}
\begin{split}
 &\sum _{g=0}^{\infty }~\sum _{\beta \in H_{2}(M,\mathbb {Z} )}{\GW}_{g,\beta} q^{\beta }\lambda ^{2g-2} \\
 = & \sum _{g=0}^{\infty }~\sum _{k=1}^{\infty }~\sum _{\beta \in H_{2}(M,\mathbb {Z} )}{\GV}_{g,\beta}{\frac {1}{k}}\left(2\sin \left({\frac {k\lambda }{2}}\right)\right)^{2g-2}q^{k\beta } ,
\end{split}
\end{equation}

In this paper, we are mainly concerned with genus zero invariants on  $X$, the quintic CY3, and the above relation can be written as
\begin{equation} \label{e:1.2}
\begin{split}
    {\GW}_{g=0, \beta= d [\op{line}]} =: {\GW}_{d} & = \sum_{e|d} \frac{1}{e^3} {\GV}_{d/e}, \\ 
    {\GV}_{g=0, \beta= d [\op{line}]} =:     {\GV}_{d} & := \sum_{e|d} \frac{\mu(e)}{e^3} {\GW}_{d/e},
\end{split}
\end{equation}
where we have used the M\"obius inversion and $\mu(e)$ is the M\"obius function. For our purpose, we will \emph{use \eqref{e:1.2} as the definition of ${\GV}_{d}$}. The first 4 terms are listed for readers' convenience.
\begin{table}[ht]
\begin{tabular}{|l|l|l|l|l|}
\hline
$d$ & 1 & 2 & 3 & 4   \\ \hline
$\GW_d$ & 2875 & 4876875/8 & 8564575000/27 & 15517926796875/64  \\ \hline
$\GV_d$ & 2875  & 609250 & 317206375 & 242467530000   \\ \hline
\end{tabular}
\end{table}

This \emph{ad hoc} definition has among other things one difficulty. Namely, the Gopakumar--Vafa invariants are to be intrinsically \emph{integers}, while the Gromov--Witten invariants are generally \emph{rational} numbers, as the above table demonstrates. Fortunately, the integrality of this definition has been shown in \cite{Ionel_Parker_2018}.
%While it is possible to verify the integrality of Gopakumar--Vafa invariants as defined by \eqref{e:1.2} by brute force computation, a more elegant approach is desirable.

There is a variant of the Gromov--Witten theory which also produces \emph{integral} invariants, namely, the \emph{quantum $K$-theory} \cite{Givental_2000, Lee_2004}.
This leads to the possibility of relating quantum $K$-invariants with Gopakumar--Vafa invariants for Calabi--Yau threefolds.

\subsection{Quantum $K$-invariants} \label{section_KGW_J}

The formulation of \emph{quantum $K$-theory} is similar to that of Gromov--Witten theory.
The $K$-theoretic Gromov-Witten invariants, or \emph{quantum $K$-invariants}, of $X$ are defined to be
\[
\langle \tau_{d_1}(\Phi_1)\dots \tau_{d_n}(\Phi_n) \rangle^{X, {K}}_{g,n,\beta} :=  \chi \Big(  \M_{g,n}(X,\beta) ; \Big(  \otimes_{i=1}^n \ev_i^*(\Phi_i) L_i^{d_i}  \Big) \otimes \O^{\vir}  \Big) \in \mathbb{Z}.
\]
Here $\Phi_1,\dots,\Phi_n \in K^0(X)$, $d_1,\dots,d_n \in \mathbb{Z}$, and $\O^{\vir}$ is the virtual structure sheaf on $\M_{g,n}(X,\beta)$ \cite{Lee_2004}. We encode all genus 0 invariants into a formal power series, called genus-0 descendant potential of $X$:
\[
F^{{K}}_0(t) = \sum_{n\geq 0}\sum_{\beta} \frac{Q^{\beta}}{n!} \langle  t(L),\dots,t(L)  \rangle^{X,{K}}_{0,n,\beta}.
\]
Here the sum is over all curve class $\beta \in H_2(X)_{\geq 0}$. %$Q$ is a formal variable, called the Novikov variable, which keeps track of the curve class. 
$t(q)$ stands for any Laurent polynomial of one variable $q$ with coefficients in $K^0(X)$
\[
 t(q) = \sum_{k\in \mathbb{Z}} \sum_{\alpha=1}^N t_k^{\alpha} \Phi_{\alpha} q^k
\]
with $\{\Phi_{\alpha} \}_{\alpha=1}^N$ a basis in $K^0(X)$. We may rewrite $F^{{K}}_0(t)$ as a more transparent form as:
\[
F^{{K}}_0(t, Q) = \sum_{n\geq 0}\sum_{\beta}\frac{Q^{\beta}}{n!} \sum_{ \substack{  k_1,\dots,k_n\in \mathbb{Z} \\ \alpha_1,\dots,\alpha_n \in \{1,\dots,N \}  }  }t_{k_1}^{\alpha_1} \cdots t_{k_n}^{\alpha_n} \langle \tau_{k_1}(\Phi_{\alpha_1}), \dots, \tau_{k_n}( \Phi_{\alpha_n} ) \rangle^{X, {K}}_{0,n,\beta}.
%F^{{K}}_0(t) = \sum_{n\geq 0}\sum_{\beta}\frac{Q^{\beta}}{n!} \sum_{ \substack{  k_1,\dots,k_n\in \mathbb{Z} \\ \alpha_1,\dots,\alpha_n \in \{1,\dots,N \}  }  }t_{k_1}^{\alpha_1} \cdots t_{k_n}^{\alpha_n} \langle L^{\otimes k_1}(\Phi_{\alpha_1}), \dots, L_n^{\otimes k_n}( \Phi_{\alpha_n} ) \rangle^{X, {K}}_{0,n,\beta}.
\]
$K^0(X)$ here stands for the topological $K$-theory of complex vector bundles. For the quintic threefold, the main focus of this paper, Grothendieck groups of coherent sheaves and complex vector bundles coincide.

\subsection{JMGS conjecture for the quintic} \label{section_conjecture}

In \cite{Jockers_Mayr_2019} and \cite{Garoufalidis_Scheidegger_2022} H.~Jockers, P.~Mayr and S.~Garoufalidis, E.~Scheidegger formulate a conjectural relation between the Gopakumar--Vafa invariants and quantum $K$-invariants for the \emph{quintic threefold} $X$.
The conjecture is formulated in terms of \emph{small $J$-functions}, a generating function in quantum $K$-theory as well as quantum cohomology.

For simplicity, let us start with the small $J^H$-function for quantum cohomology \cite{Givental_1996}
\footnote{The peculiar evaluation at $0$ is a reminder that the small $J$-function is a restriction of the \emph{big $J$-function}. Cf.\ Definition~\ref{def:KbigJ}.}
\[
\begin{split}
J^H(t=0, z, Q) & := -z + \sum_{d \geq 1, \alpha} Q^d \langle \frac{\phi_{\alpha}}{-z-\psi} \rangle^H_{0,1,d} \phi^{\alpha} 
\\
&= -z + \sum_{d \geq 1} Q^d \Big( - \frac{d \GW_d}{5z} H^2 - \frac{2 \GW_d}{5z^2} H^3  \Big),
\end{split}
\]
where $H$ is the hyperplane class pulled back from $\PP^4$. 
It follows directly from the divisor and dilaton equations that
\[
\langle \rangle^H_{0,0,d}=: \GW_d  , \quad \langle \psi \rangle^H_{0,1,d} = -2 \GW_d, \quad \langle H \rangle^H_{0,1,d} = (\int_d H) \GW_d = d \cdot \GW_d.
\]

To phrase the conjecture in $K$-theory, recall that 
\[
K^0(X) = \frac{\mathbb{Q}[P]}{(1-P)^4},
\]
where $P= \O(-1)|_X$. 
We fix a basis $\{ \Phi_{\alpha} \}_{\alpha=0}^3= \{ (1-P)^{\alpha} \}_{\alpha=0}^3$ for $K^0(X)$. The inner product reads:
\[
(\Phi_a, \Phi_b)^{K} : = \chi (X, \Phi_a \Phi_b) 
= \left(
\begin{matrix}
0 & 5 & -5 & 5
\\
5 & -5 & 5 & 0
\\
-5 & 5 & 0 & 0
\\
5 & 0 & 0 & 0
\end{matrix}
\right).
\]
The dual basis $\{ \Phi^{\alpha} \}_{\alpha=0}^3$ is given by
\begin{align*}
    \Phi^0 &= \frac{1}{5}(1-P)^3  , &\quad\Phi^1 = \frac{1}{5}((1-P)^2 + (1-P)^3), \\
 \Phi^2 &= \frac{1}{5} ((1-P) + (1-P  )^2) , & \quad \Phi^3 = \frac{1}{5}(1+(1-P) - (1-P)^3).
\end{align*}

The small $J^K$-function for quantum $K$-theory is defined similarly. Jockers--Mayr and Garoufalidis--Scheidegger conjecture that $J^K$ can be expressed as a \emph{linear} combination of GV invariants.
\begin{conjecture} [{\cite{Jockers_Mayr_2019, Garoufalidis_Scheidegger_2022}} ]  \label{conjectureJK}
\[
\begin{split}
 & \frac{1}{1-q} \left[ J^{K}(0, q, Q)\right] := \frac{1}{1-q} \left[  (1 -q) +  \sum_{\alpha} \sum_{M \geq 1}  \Phi_{\alpha} \langle \frac{\Phi^{\alpha}}{1-qL} \rangle^K_{0,1,M} Q^M \right]
\\
 = & 1+ (1-P)^2 \sum_{d,r \geq 1} a(d,r,q^r) \GV_d Q^{dr} + (1-P)^3 \sum_{d,r \geq 1} b(d,r,q^r) \GV_d Q^{dr},
\end{split}
\]
where
\[
\begin{split}
    5 a(d,r,q) &= \frac{d(r-1)}{1-q} + \frac{d}{(1-q)^2},
    \\
    5 b(d,r,q) &= \frac{rd+r^2-d-1}{1-q} + \frac{d+3}{(1-q)^2} -\frac{2}{(1-q)^3}.
\end{split}
\]
\end{conjecture}

The main purpose of this paper is to prove the above conjecture.
%We will first review fundamental works by A.~Givental and his collaborators on the structure of quantum $K$-theory and its relation with quantum cohomology. 

\section{Review of Givental's framework of quantum $K$-theory} \label{s:2}
In this section, relevant results from \cite{Givental_Tonita_2011, Givental_ER_Coh_K_2015, Givental_PEV_Toric_2015, Givental_PEVIII_ER} will be recalled.
In these works and many others, A.~Givental and his collaborators have established groundbreaking works on the structural understanding of quantum $K$-theory, including a characterization of the genus zero quantum $K$-theory in terms of quantum cohomology. This is essential for our proof as $\GV_d$ in this article are actually defined in terms of $\GW$ by \eqref{e:1.2}.

%From now on, \emph{$X$ stands for the quintic Calabi--Yau threefold}.

\subsection{The symplectic loop space formalism}
Let $\Lambda = \mathbb{Q}[\![Q]\!]$ be the Novikov ring.
Givental's loop space for quantum $K$-theory is defined as
\[
\K := K^0(X)(q) \otimes \Lambda.
\]
By definition, an element in $\K$ is a $Q$-series with coefficients rational function of $q$. 
\begin{remark} \label{Element_in_K}
Let $K :=K^0(X) \otimes \Lambda$. An element in $\K$ can also be understood as a rational function of $q$ with coefficients in $K$ in the $Q$-adic sense, i.e. modulo any power of the maximal ideal in the Novikov ring. %We write $f(q) \in \K$ with this abuse of notation.
\end{remark}

$\K$ is quipped with a symplectic form $\Omega$, i.e., the $\Lambda$-valued non-degenerate anti-symmetric bilinear form:
\[
\K \ni f,g \rightarrow \Omega(f,g) := \Big( \Res_{q=0} + \Res_{q=\infty} \Big) ( f(q), g(q^{-1}) )^{K} \frac{dq}{q},
\]
where $( \cdot, \cdot )^{K}$ denote the $K$-theoretic intersection pairing on $K$:
\[
(a,b)^{K} := \chi(X, a \otimes b) = \int_X \td (T_X) \ch(a) \ch(b).
\]
$\K$ admits the following Lagrangian polarization with respect to $\Omega$:
\[
\begin{split}
    \K &= \K_+ \oplus \K_-
    \\
    & := K[q,q^{-1}] \, \oplus \, \{ f(q) \in \K | f(0) \neq \infty ,\ f(\infty) =0 \}.
\end{split}
\]
\begin{definition}\label{def:KbigJ}
The big $J$-function of $X$ in the $K$-theory is defined as a map $\K_+ \rightarrow \K$:
\[
\begin{split}
   t \mapsto &J^{\KK}(t) := J^{\KK}(t, q, Q) \\
   &:= (1-q) + t(q) + \sum_{\alpha} \Phi_{\alpha}\sum_{n,d} \frac{Q^d}{n!} \langle \frac{\Phi^{\alpha}}{1-qL}, t(L),\dots,t(L) \rangle ^{X}_{0,n+1,d},
\end{split}
\]
where $\{ \Phi_{\alpha} \}$ and $\{\Phi^{\alpha}\}$ are Poincar\'e-dual basis of $K^0(X)$ with respect to $( \cdot, \cdot )^{K}$. 
\end{definition}

\begin{conventions}
In this series of papers, the variables $q$ and $Q$ in $J^{\KK}(t, q, Q)$  will stay in the background and the $J$-function is often denoted by $J^{\KK}(t)$.
\end{conventions}

\begin{remark} \label{J-function}
The definition of $J$-function contains three ingredients. The first summand, $1-q$, is called the \emph{dilaton shift}. The second one, $t(q)$, is call the \emph{input}. These two lie in $\K_+$.

We claim that the last term in Definition~\label{def:KbigJ}, a sum of correlators, lies in $\K_-$. What needs to be shown is that it is a rational function of $q$, with no pole at $q=0$ and with a zero at $q=\infty$. On each particular moduli stack,  there exists a polynomial $P(q)$ such that $P(L^{-1})=0$ and $P(0) \neq 0$ since $L^{-1}$ is invertible.
%with zero at some roots of unity. 
On such a moduli space, we have $P(q)-P(L^{-1}) = F(q,L) (L^{-1} -q)$ for some $F(q, L)$, which is a polynomial in $q$ with $\deg_q F < \deg_q P$. Thus, %expand $\displaystyle {(1-qL)^{-1}}$ as follows:
\[
\frac{1}{1-qL} = \frac{\frac{P(q) - P(L^{-1})}{P(q)}}{1-qL} = \frac{\frac{P(q) - P(L^{-1})}{1-qL}}{P(q)} = \frac{L^{-1} F(q,L)}{P(q)},
\] %Note also that $P(0) \neq 0$ since $L$ is invertible. 
We conclude that each correlator is a rational function in $q$, has no pole at $q=0$, vanishes at $q=\infty$ and hence lies in $\K_-$. 

Note that this last term (in $\K_-$) actually has poles only at roots of unity, cf.\ Remark~\ref{r:2.14}.
\end{remark}

%The following proposition explains why $J$-function is an important object to study. 
Let 
\[
\mathcal{L}^{K} := \{ ( 1-q + t(q), J^{K}(t)|_{\K_-} ) \in \K_+ | t(q) \in \K_+      \}
\]
be the graph of the big $J^K$-function in $\K$. We have
\begin{proposition} [{\cite[\S~2,\S~3]{Givental_Tonita_2011}}]
The Lagrangian submanifold $\mathcal{L}^{K}$ has the following properties:
\begin{enumerate}
    \item $\mathcal{L}^{K}$ is cone over $\Lambda$; % That is, it is preserved under the action of $\Lambda$.
    \item For $f\in \mathcal{L}^{K}$, we have $T_f \mathcal{L}^{K} \cap \mathcal{L}^K = (1-q)T_f \mathcal{L}^{K}$; Indeed, $\mathcal{L}^{K}$ is ruled by a finite-parameter family of such spaces.
 %   \item the set of all tangent spaces to $\mathcal{L}^{K}$ forms a finite-dimensional family; thus $\mathcal{L}^{K}$ is ruled by a finite-dimensional family of linear subspaces;
    \item[(3)] Identifying $T^* \K_+$ with $\K = \K_+ \oplus \K_-$ as symplectic spaces, we have
    $$J^{K}(q) = (1-q) + t(q) + dF_0^{{K}}.$$
\end{enumerate}
\end{proposition}
%\begin{remark}
%The dilaton shift $1-q$ is important. $\mathcal{L}^{K}$ is not even a cone without the shift.

%The first two properties together is called overruled in the literature. In \cite{Givental_2004_Symplectic}, Givental proves that overruled condition is equivalent to three fundamental results in Gromov–Witten theory: the string equation, the dilaton equation and the topological recursion relations. 

%For (3), note that $T^*\K_+$ is equipped with a natural symplectic structure. The identification is given by identifying their $\K_+$ part and the remaining part follows by the polarization. Notice that changing the polarization in $\K$ changes the identification and implicitly changes the Gromov-Witten invariants.
%\end{remark}

\subsection{Fake quantum $K$-theory}
The fake quantum $K$-invariants of $X$ are defined to be
\[
\langle \tau_{d_1}(a_1)\dots \tau_{d_n}(a_n) \rangle^{X, \fake}_{g,n,d} := \int_{[\M_{g,n}(X,d)]^{\vir} } \td(T_{\M_{g,n}(X,d)}^{\vir}) \ch ( \otimes_{i=1}^n \ev_i^*(a_i) L_i^{d_i} ),
\]
where $[\M_{g,n}(X,d)]^{\vir}$ is the (cohomological) virtual fundamental class and $[T^{\vir}]$ the virtual tangent bundle, the virtual difference of deformation and obstruction.

In the fake quantum $K$-theory, $L_i$ are unipotent. We expand $\displaystyle \frac{1}{1-qL}$ as follows:
\[
\frac{1}{1-qL} = \sum_{k \geq 0} \frac{q^k}{(1-q)^{k+1}}  (L-1)^k.
\]
This motivates the following definitions of symplectic loop space and $J_X^{\fake}$. Let 
\[
\K^{\fake} := \{ \text{$Q$ power series with coefficients Laurent series in $1-q$}   \} .
\]
%be the space of Laurent series in $1-q$ with coefficient in $K$ in the $Q$-adic sense. 
%See Remark \ref{Element_in_K}.
%
$\K^{\fake}$ is endowed with the symplectic form $\Omega^{\fake}$
\[
\Omega^{\fake}(f,g) := - \Res_{q=1} (f(q), g(q^{-1}))^{K} \frac{dq}{q},
\]
and a corresponding Lagrangian polarization:
\[
\begin{split}
    \K^{\fake} &= \K_+^{\fake} \oplus \K_-^{\fake}
    \\
 : &= K[\![q-1]\!] \oplus {\rm span}_K \left\{ \frac{q^k}{(1-q)^{k+1}}  \right\}_{k \geq 0}.
\end{split}
\]
\begin{definition}
The big $J$-function of $X$ in the fake $K$-theory is defined as a map $\K_+^{\fake} \rightarrow \K^{\fake}$:
\[
t \rightarrow J^{\fake} (t) := (1-q) + t(q) + \sum_{\alpha} \Phi_{\alpha} \sum_{n,d} \frac{Q^d}{n!} 
\langle \frac{\Phi^{\alpha}}{1-qL}, t(L),\dots,t(L) \rangle ^{X,\fake}_{0,n+1,d}.
\]
\end{definition}
Same as in Remark \ref{J-function}, $J^{\fake}$ has dilation shift $(1-q)$, input $t(q)$, and a sum of correlator lying in $\K_-^{\fake}$. The graph of the $J^{\fake}$-function is a Lagrangian cone $\mathcal{L}^{\fake} \subset \K^{\fake}$, which is overruled in the sense that its tangent space $T$ are tangent to $\mathcal{L}^{\fake}$ along $(1-q)T$.

To describe the relation between fake $K$-theory and cohomological theory, 
%consider the (cohomological) Gromov-Witten invariants of $X$:
%\[
%\langle \tau_{d_1}(\gamma_1)\dots \tau_{d_n}(\gamma_n) \rangle^{X,H}_{g,n,d} :=  \int_{ [\M_{g,n}(X,\beta)]^{\vir} } \prod_{i=1}^n \Big(\ev_i^*(\gamma_i) \psi_i^{d_i} \Big) .
%\]
take $H = H(X)[\![Q]\!]$, with Poincar\'e pairing $(a,b)^H = \int_X a\cdot b$. The Givental's loop space
\[
\mathcal{H} := H [\![ z, z^{-1}],
\]
is endowed with a symplectic form $\Omega^H$:
\[
\Omega (f,g) := \Res_{z=0} (f(-z), g(z) )^H dz,
\]
and a corresponding Lagrangian polarization:
\[
\begin{split}
    \mathcal{H} &= \mathcal{H}_+ \oplus \mathcal{H}_-
    \\
    & = H[\![z]\!] \oplus z^{-1} H[z^{-1}].
\end{split}
\]
\begin{definition}
The big $J$-function of $X$ in the cohomology theory is defined as a map $\mathcal{H}_+ \rightarrow \mathcal{H}$:
\[
t \rightarrow J^H(t) := -z + t(z) + \sum_{\alpha} \phi_{\alpha} \sum_{n,d} \frac{Q^d}{n!} 
\langle \frac{\phi^{\alpha}}{-z-\psi_1}, t(\psi_2),\dots,t(\psi_{n+1}) \rangle ^{X,H}_{0,n+1,d},
\]
where $\{ \phi_{\alpha} \}$ and $\{ \phi^{\alpha} \}$ are Poincar\'e dual basis in $H(X)$ with respect to $( \cdot, \cdot )^H$. We denote the graph of the big $J^H$-function as $\mathcal{L}^H$. 
\end{definition}

%As in Remark \ref{J-function}, we have dilaton shift $-z$, input $t(z)$, and the sum of correlator terms lies in $\mathcal{H}_-$. The graph of the $J^{H}$-function is a Lagrangian cone $\mathcal{L}^{H} \subset \mathcal{H}$, which is overruled in the sense that its tangent space $T$ are tangent to $\mathcal{L}^{H}$ along $zT$.

Introduce the quantum Chern character by
\[
\begin{split}
    {\rm qch}: \K^{\fake} &\rightarrow \mathcal{H}
    \\
    \sum_k f_k (q-1)^k &\rightarrow \sqrt{ \td(TX) } \sum_k \ch (f_k) (e^z-1)^k.
\end{split}
\]
It is symplectic, i.e. qch$^*\Omega^H = \Omega^{\fake}$.
\begin{theorem}[\cite{Coates_Givental_2007}] \label{thm_coh_fake}
Denote by $\Delta$ the Euler–Maclaurin asymptotic of the infinite product
\[
\Delta \sim \sqrt{ \td(T_X -1) } \prod_{m=1}^{\infty} \td ((T_X-1) \otimes L_z^{-m}),
\]
where $L_z$ is the universal line bundle with first chern class $z$. Then $\mathcal{L}^{\fake}$ can be obtained by $\mathcal{L}^H$ by the multiplication of $\Delta$:
\[
{\rm qch} (\mathcal{L}^{\fake}) = \Delta \mathcal{L}^H.
\]
\end{theorem}

\subsection{$J^K$-function as a graph sum via Kawasaki's HRR} \label{section_KHRR_recursion}
Given a nonsingular projective variety $M$ with a locally free sheaf $E$, the Hirzebruch-Riemann-Roch formula expresses the Euler characteristic in terms of a cohomological integration:
\[
\chi(M;E) := \sum_{i = 0}^{\dim M} (-1)^i h^i(M;E) = \int_M \td(TM) \cdot \ch(E).
\]
A generalization to smooth Deligne-Mumford stacks $\mathcal{M}$ %with a vector bundle $E$ is due to 
was first established by T.~Kawasaki \cite{Kawasaki_1979} in the analytic setting. 

Let $I\mathcal{M} = \sqcup_i \mathcal{M}_i$ be the inertia stack of $\mathcal{M}$, with $\mathcal{M}_i$ connected components. Following Givental, we refer to them as \emph{Kawasaki strata}. Kawasaki's formula reads
\[
\chi(\mathcal{M}, E) = \sum_{i} \frac{1}{m_i} \int_{\mathcal{M}_i} \td(T_{\mathcal{M}_i}) \ch \Big(  \frac{ \Tr (E|_{\mathcal{M}_i}) }{ \Tr ( \Lambda^* N_{\mathcal{M}_i}^* ) } \Big),
\]
where the sum over $i$ runs through all connected components.

As in \cite{Givental_Tonita_2011}, a (virtual) version of the above formula, valid for quasi-smooth DM stacks, will be used to compute the (big) $J$-function. Kawasaki's strata are visualized in Figure \ref{Kawasaki_Strata} as in \cite[\S~7]{Givental_Tonita_2011}.

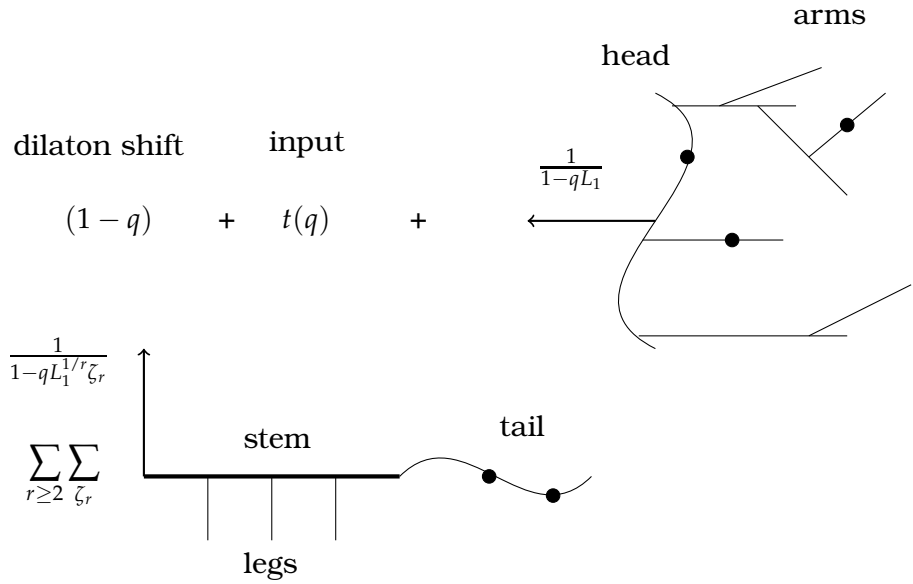
\begin{figure}[ht]
\centering
\begin{tikzpicture}[scale=1.7]
\filldraw[black] (0.9,-1.4) node[anchor=west]{dilaton shift};
\filldraw[black] (2.9,-1.4) node[anchor=west]{input};
\filldraw[black] (5.5,-0.7) node[anchor=west]{head};
\filldraw[black] (7,-0.4) node[anchor=west]{arms};

\filldraw[black] (1.3,-2) node[anchor=west]{$(1-q)$};
\filldraw[black] (2.5,-2) node[anchor=west]{+};
\filldraw[black] (3,-2) node[anchor=west]{$t(q)$};
\filldraw[black] (4,-2) node[anchor=west]{+};

\filldraw[black] (5,-1.6) node[anchor=west]{$\frac{1}{1-qL_1}$};
\draw [thick ,-to] (6,-2) -- (5,-2);

\draw (6,-1) .. controls (7,-1.5) and (5,-2.5) .. (6,-3);
\draw (6.13,-1.1) -- (7.1,-1.1);
\draw (6.5, -1.1) -- (7.3, -0.8);
\draw (7.5,-1.8) -- (6.8,-1.1);
\draw (7.2, -1.5) -- (7.8, -1);
\draw (5.9,-2.15) -- (7,-2.15);
\draw (5.87, -2.9) -- (7.5, -2.9);
\draw (8, -2.5) -- (7.2, -2.9);

\filldraw[black] (6.25,-1.5) circle (1.5pt);
\filldraw[black] (6.6,-2.15) circle (1.5pt);
\filldraw[black] (7.5,-1.25) circle (1.5pt);

\filldraw[black] (1,-4) node[anchor=west]{$\displaystyle\sum_{r\geq 2}\sum_{\zeta_r}$};
\filldraw[black] (0.85,-3.1) node[anchor=west]{$\frac{1}{1-qL_1^{1/r}\zeta_r} $};

\filldraw[black] (2.7,-3.7) node[anchor=west]{stem};
\draw[ultra thick] (2,-4) -- (4,-4);
\draw[thick, -to] (2,-4) -- (2,-3);
\draw (2.5,-4) -- (2.5,-4.5);
\draw (3,-4) -- (3,-4.5);
\draw (3.5,-4) -- (3.5,-4.5);
\filldraw[black] (2.7,-4.7) node[anchor=west]{legs};

\draw (4,-4) .. controls (4.5,-3.5) and (5,-4.5) .. (5.5,-4);
\filldraw[black] (4.7,-3.6) node[anchor=west]{tail};
\filldraw[black] (4.7,-4) circle (1.5pt);
\filldraw[black] (5.2,-4.15) circle (1.5pt);
\end{tikzpicture}
\caption{Kawasaki strata}
\label{Kawasaki_Strata}
\end{figure}

Notations of Figure \ref{Kawasaki_Strata} are explained below. A point in the inertia stack is represented by a stable map with symmetry, which we denote by $g$. 
%\marginpar{This figure and many equations in Section 3 and Appendix are overfull! Make sure that it is all right with the thesis. Might have to use ``small font'' or break the equations.}
\begin{itemize}
    \item The black points denote the marked points. 
    \item Given a stable map with symmetry $g$, the action of $g$ fixes the marked points and acts on $L_1$ with an eigenvalue, which is denote by $\zeta$. The strata with $\zeta=1$ are separated from those where $\zeta_r \neq 1$ in the Figure \ref{Kawasaki_Strata}. $\displaystyle \sum_{\zeta_r}$ denotes the sum over all primitive $r$-th roots of unity.
    \item When $\zeta=1$, the \emph{head} is a stable map with its source curve the maximal connected subcurve
    containing the first marked points where $g$ is trivial. The head has marked points which are either the marked points of the original curve or nodal points attached to arms.
    \item An \emph{arm} is a stable map whose source curve is obtained as a connected component of the original curve when the head is removed. It has its own ``first marked point'', the nodal point where it is attached to the head. 
    
    By definition, an arm can have any symmetry satisfying the following condition: the eigenvalue of the symmetry at its ``first marked point'' $\neq 1$, or else it is part of the head.
    
    \item When $\zeta=\zeta_r \neq 1$, we consider the stem curve defined in the next paragraph. It is a connected subcurve, containing the first marked point with $g$ as a symmetry of order $r$. The quotient stable map is called the \emph{stem}. 
    
    The stem curve contains the maximal chain of $\PP^1$ containing the \emph{first marked point} and with the same action of the symmetry on each of them. In other words, a node connects two $\PP^1$'s should be balanced, i.e., the eigenvalues of $g$ on the two branches are inverse. Furthermore, it can have other irreducible components attached on the “side” of the chain permuted by $g$. The opposite end of the maximal chain of $\PP^1$'s is either a node or a marked point. In either case, we refer it as the \emph{``last marked point''}.
    
    The stem carries two kinds of marked points:
    \begin{enumerate}
        \item Ramified/stacky points at the first and last marked points. The input at the first marked point is ${\displaystyle (1-qL_1^{1/r} \zeta_r)^{-1}}$.
        \item Unramified/smooth marked points coming from symmetric configurations of $r$-tuples of nodes on the cover. (These nodes connect to legs; see below.)
    \end{enumerate}

    %  \begin{enumerate}
    %    \item The ramified one on the stacky points of source curve, i.e. the (smooth) points fixed by $g$. There are two such points. One will be the first marked point and has input
     %   \[
      %  \frac{1}{1-qL_1^{1/r} \zeta_r}
       % \]
    %    as shown in Figure \ref{Kawasaki_Strata}, where $L_1$ denote the cotangent line bundle on the moduli space of stems.
    %    Indeed, let $L_1'$ be the universal cotangent line to the original stable map. In the Kawasaki's formula, we have
     %   \[
     %   \ch (\Tr L_1') = \zeta_r e^{c_1(L_1)/r}.
     %   \]

     %   The other stacky points is at the end of the chain with eigenvalue $\zeta_r^{-1}$ on the cotangent line by the action of $g$. If it is marked, we will assign it to be the last marked point.
      %  \item The unramified one, which come from symmetric configurations of $r$-tuples of nodes on the cover.
 %   \end{enumerate}

    \item A \emph{tail} is a stable map whose source curve is obtained as \emph{the} connected component of the original curve after removing the stem curve and is attached the last marked point. %It attaches to the point fixed by $g$ and with eigenvalue $\zeta^{-1}$ on the cotangent line. 
    Its ``first marked point'', the nodal point where it is attached to the stem, is fixed by the $g$ action with the eigenvalue $\neq \zeta$, (or else it is part of the stem.)
    
    \item \emph{Legs} are connected components after removing the stem curve and the tail. They are permuted by the action of $g$.
\end{itemize}

The big $J$-function can be rewritten as
\[
\begin{split}
J^{\KK}(t) &=(1-q) + t(q) + \sum_{n,d,\alpha} \frac{Q^d \Phi_{\alpha}}{n!}\langle \frac{\Phi^{\alpha}}{1-qL}, t(L),\dots,t(L)  \rangle^{X,\KK}_{0,n+1,d}
\\
&= (1-q) + t(q) + \sum_{n,d,\alpha} \frac{Q^d \Phi_{\alpha}}{n!}\sum_{\zeta} \langle \frac{\Phi^{\alpha}}{1-qL}, t(L),\dots,t(L)  \rangle^{X_{\zeta}}_{0,n+1,d},
\end{split}
\]
where $\sum_{\zeta}$ runs through all roots of unity (including 1), $X_{\zeta}$ is the  Kawasaki strata where $g$ acts on $L_1$ with eigenvalue $\zeta$, and $\langle \rangle^{X_{\zeta}}$ denotes the contribution of $X_{\zeta}$ in the Kawasaki's formula. In other words, $\langle \rangle^{X,\KK}$ represent (true) quantum $K$-invariants while $\langle \rangle^{X_{\zeta}}$ stand for (collections of) \emph{cohomological} invariants. Explicit formulas will be reorganized in Proposition~\ref{Prop_Stem_inv}.

Let arm$(L)$, leg$_{\zeta}(L)$, and tail$_{\zeta}(L)$ denote the totalities of the contributions from arm, leg, and tail respectively. They can be computed as follows:
\[
\begin{split}
    {\rm arm}(q) &= \sum_{n,d \neq 0,\alpha} \frac{Q^d \Phi_{\alpha}}{n!}\sum_{\zeta' \neq 1} \langle \frac{\Phi^{\alpha}}{1-qL}, t(L),\dots,t(L)  \rangle^{X_{\zeta'}}_{0,n+1,d} 
    \\
    {\rm leg}_{\zeta}(q) &= \Psi^r ({\rm arm}(q)_{t=0} )
    \\
    {\rm tail}_{\zeta}(q) &= \sum_{n,d \neq 0,\alpha} \frac{Q^d \Phi_{\alpha}}{n!}\sum_{\zeta' \neq \zeta} \langle \frac{\Phi^{\alpha}}{1-qL}, t(L),\dots,t(L)  \rangle^{X_{\zeta'}}_{0,n+1,d} ,
\end{split}
\]
where $\zeta$ is the primitive $r$-th roots of unity, $\zeta'$ in the above sums are arbitrary roots of unity and
$\Psi^r$ are the \emph{Adams operations}. Recall that Adams operations are additive and multiplicative endomorphisms of $K$-theory or more generally $\lambda$-rings, acting on line bundles by $\Psi^r(L) = L^{\otimes r}$. Here $\Psi^r$ also act on the Novikov variables by $\Psi^r(Q^d) = Q^{rd}$. In the definition of leg$_{\zeta}(q)$, $t=0$ is imposed because legs will be permuted by $g$, an automorphism, which excludes the appearance of marked points.

The following propositions (\ref{prop_head_fake}, \ref{Prop_Stem_inv}) justify the above definitions.

\begin{proposition}[Essentially {\cite[Proposition~1]{Givental_Tonita_2011}}]\label{prop_head_fake}
\[
J^{\KK}(t)|_{q=1} = J^{\fake}(t + {\rm arm}),
\]
where $(\cdot )|_{q=1}$ is the power series expansion at $q=1$ (of the rational function). 
\end{proposition}
%\begin{proof}
%Denote by $L_-$ the cotangent line at a marked point of the head. If it is the marked point of the original curve, the contribution is $t(L_-)$. When it is a node where an arm is attached, denote by $L_+$ the cotangent line to the arm. Note that smoothing the node gives the factor
%\[
%\frac{ \sum_{\alpha}\Phi_{\alpha}\otimes \Phi^{\alpha} }{1-L_- L_+},
%\]
%and that the arm sums up contributions of all possible curves except those with $\zeta=1$. It is exactly the definition of arm$(q)$ with $q$ replaced by $L_-$. 

%Finally, note that the head can be any curve with identity as its symmetry. This finish the proof.
%\end{proof}

For $\zeta\neq 1$ a primitive $r$-th roots of unity, the stem space is isomorphic to the moduli space
%${\displaystyle \M_{0,n+2} \left( [X/\mathbb{Z}_r], d; (g,1,\dots,1,g^{-1}) \right) }$, denoted by 
\[
 \M_{0,n+2,d}^X(\zeta) := \M_{0,n+2} \left( [X/\mathbb{Z}_r], d; (g,1,\dots,1,g^{-1}) \right) .
\]
% henceforth for short. 
Here the group elements $g, 1$ signal the twisted sectors in which the marked points lie. To simplify the notation, we introduce the notation
\begin{equation} \label{e:kawasaki}
\begin{split}
    &\Big[ T_1(L), T(L),\dots, T(L), T_{n+2}(L) \Big]^{X_{\zeta}}_{0,n+2,d}
    \\
 := &\int_{[\M_{0,n+2,d}^X(\zeta)]^{\vir}} \td( T_{\M} ) \ch \left( 
    \frac{ \ev_1^* (T_1(L)) \ev_{n+2}^* (T_{n+2}(L)) \prod_{i=2}^{n+1} \ev_i^* T(L) }
    {\Tr (\Lambda^* N^*_{\M})}  \right),
\end{split}
\end{equation}
where $[\M_{0,n+2,d}^X(\zeta)]^{\vir}$ is the virtual fundamental class, $T_{\M}$ is the (virtual) tangent bundle to $\M_{0,n+2,d}^X(\zeta)$, and $N_{\M}$ is the (virtual) normal bundle of $\M_{0,n+2,d}^X(\zeta)$ considered as Kawasaki strata in $\M_{0,nr+2}(X,rd)$. 

\begin{proposition} [{ \cite[\S7]{Givental_Tonita_2011} }] \label{Prop_Stem_inv}
Let $\zeta$ be a primitive $r$-th roots of unity.
We have
\[
\begin{split}
    & \sum_{n,d,\alpha}  \frac{Q^d \Phi_{\alpha}}{n!}\langle \frac{\Phi^{\alpha}}{1-qL}, t(L),\dots,t(L)  \rangle^{X_{\zeta}}_{0,n+1,d} 
    \\
   = & \sum_{n,d,\alpha} \frac{Q^{rd}\Phi_{\alpha}}{n!} \Big[ \frac{\Phi^{\alpha}}{1-q\zeta L^{1/r}}, {\rm leg}_{\zeta}(L),\dots, {\rm leg}_{\zeta}(L),  \delta_{\zeta}(L^{1/r}) \Big]^{X_{\zeta}}_{0,n+2,d},
\end{split}
\]
where 
\[
\delta_{\zeta}(q) = (1- \zeta^{-1} q) + t(\zeta^{-1} q) + {\rm tail}_{\zeta}(\zeta^{-1} q).
\]
\end{proposition}
\begin{proof}[Sketch of Proof]
The statement is scattered in \cite[\S7]{Givental_Tonita_2011}. A sketch of the proof is given below.

The factor $\displaystyle\frac{1}{1-q\zeta L^{1/r}}$, coming from the ``ramified'' first marked point on a stem curve, with $\zeta$ the eigenvalue of the $g$-action. The power $1/r$ on $L$ comes from the comparison of the cotangent line between the marked point on the stem and its $r$-fold cover.

The contributions at unramified marked points can be written in terms of legs, where the following lemma has been used.

\begin{lemma} [{\cite[\S 7 Lemma]{Givental_Tonita_2011} }] \label{lemma_Adams_cyclic}
Let $V$ be a vector bundle, and $g$ the automorphism of $V^{\otimes r}$ acting by the cyclic permutation of the factors. Then
\[
\Tr (g| V^{\otimes r}) = \Psi^r(V). 
\]
\end{lemma}

The input $\delta(q)$ in the last marked point contains three parts. They come from the following cases:
\begin{itemize}
    \item If the last marked point is a regular point on the covering curve, the infinitesimal translation of that point gives the factor $1-\zeta^{-1}L^{1/r}$.
    \item If the last marked point is also a marked point on the covering curve, the input $t(\zeta^{-1} L^{1/r})$ is kept. 
    \item If the last marked point is the node with a tail attached on the cover, the contribution is tail$_{\zeta} (\zeta^{-1} L^{1/r}) $, since the tail can be any stable map with eigenvalue at the first marking $\neq \zeta^{-1}$. This is exactly the definition of tail$_{\zeta}(q)$
\end{itemize}
%In all 3 cases, $\zeta^{-1}$ comes from the action of $g$ and the power $1/r$ comes from the comparison of the cotangent line between the marked point on the stem and its $r$-fold cover.
As for the first marked point, the nontrivial action of $g$ creates the $r$-the roots $\zeta$ and $L^{1/r}$.
This finish the proof.
\end{proof}

\begin{remark} \label{r:2.14}
Since $L^{1/r}$ is unipotent, %the following expansion can be performed
\[
\frac{1}{1-q\zeta L^{1/r}} = \sum_{i \geq 0} \frac{ (\zeta q)^i  }{ (1-\zeta q)^{i+1} } (L^{1/r} -1)^i.
\]
This expresses $J^{\KK}|_{\K_-}$ as sums of rational functions with poles only at roots of unity. Cf.\ Remark~\ref{J-function}.
\end{remark}

\subsection{Reconstruction theorem} \label{s:2.5}
We now specialize to the quintic $X$.
There are several reconstruction theorems that determine parts of the overruled Lagrangian cone $\mathcal{L}^{\KK}$ from certain initial values. See \cite{Lee_Pandharipande_2004, Givental_ER_Coh_K_2015, Givental_PEVIII_ER, Iritani_Milanov_Tonita}.

A convenient choice of the initial values for the reconstruction theorem is the Givental's $I^K$-function for $X$. First consider \cite{Givental_Lee_2003}
\[
J^K_{\PP^4} (0) = (1-q) \sum_{d \geq 0} \frac{Q^d}{ \prod_{k=1}^d (1-Pq^k)^5 },
\]
where $P = \O(-1)$ and hence $(1-P)^5=0$. Then \cite{Givental_PEV_Toric_2015} gives
\[
I^K = \sum_{d\geq 0} I_d^K Q^d = (1-q) \sum_{d=0}^{\infty} Q^d \frac{\prod_{k=1}^{5d} (1-P^5q^k) }{\prod_{k=1}^d (1-Pq^k)^5} = J^K(t^*),
\]
where the input $t^*$ can be computed:
\[
t^*(q) = I^K |_{\K_+}  -(1-q).
\]

\begin{theorem}[{Essentially \cite[Theorem~2]{Givental_PEVIII_ER}}] \label{t:2.12}
For $X$ the quintic threefold, we have
\begin{equation} \label{e:2.1}
\begin{split}
    &J^K(0) =  \sum_{d\geq 0}  I_d^K Q^d \cdot \\ & \quad  \exp \left( \sum_{k>0}\frac{ \sum_{i=0}^3 \Psi^k(\e_i(Q)) (1-P^kq^{kd})^i}{k(1-q^k)} \right) \sum_{i=0}^3 r_i(q,Q) (1-Pq^d)^i,
\end{split}
\end{equation}
for some uniquely determined $\e_i(Q)$ and $r_i (q, Q)$, where
\begin{equation} \label{e:2.2}
\begin{split}
    \e_i(Q) = \sum_{j \geq 1} \e_{ij}Q^j \in \mathbb{Q}[\![Q]\!], \\
    r_i(q,Q) %= \sum_{j,k\geq 0} r_{ijk} q^jQ^k 
    = \sum_{j\geq 0} r_{ij}(q) Q^j \in \mathbb{Q}[q][\![Q]\!].
\end{split}
\end{equation}
\end{theorem}
\begin{proof}[Sketch of proof]
This is essentially \cite[Theorem~2]{Givental_PEVIII_ER}, with minor improvements due to the special properties of the quintic. Givental's reconstruction \eqref{e:2.1} holds for some $\e_i$ and $r_i(q) \in \mathbb{Q}[q, q^{-1}][\![Q]\!]$. Using the induction on the degree of $Q$ and the condition 
\[
J^K(0) |_{\K_+} \left( := \text{projection to $\K_+$ of $J^K (0)$} \right) = 1-q,
\]
$\e_i(Q)$ and $r_i(q,Q)$ are uniquely determined. More precisely, given any $M$, the coefficients of $Q^M$ of the projections of the left hand side and of the right hand side of \eqref{e:2.1} to $\K_+$ must be equal
\begin{equation} \label{e:2.3}
\begin{split}
M=0, \quad (1-q) = (1-q) \sum_{i=0}^3 r_{i0} (1-P)^i \,  \Rightarrow \,  \sum_{i=0}^3 & r_{i0} (1-P)^i =1, \\
 \sum_i \left[ \e_{iM} + (1-q) r_{iM}(q) \right] (1- P)^i  = \sum_i f_i(q) (1-P)^i \, & \Rightarrow \, \e_{iM} = f_i(1)
\end{split}
\end{equation}
for some polynomials $f_i(q)$ of $q$, which can be written as polynomials of $\e_{i'M'}$, $r_{i'M'}$ and (components of) $I_{M'}^K$ for $M'<M$. This in particular also implies, by induction, that \emph{$r_{ik}(q)$ are polynomials of $q$} (instead of Laurent polynomials).
\end{proof}

\section{Proof of the conjecture} \label{section_proof}
\subsection{Outline of the proof}
As explained in the Introduction our proof involves both the virtual orbifold Hirzebruch--Riemann--Roch and the explicit reconstruction theorem, and is a little complicated. For readers' convenience, we break it down in four steps, and relegate some computations to Section~\ref{appen_Kawasaki} and Appendix~\ref{Appendix_degree0} so as not to break the flow of the proof.
%The first two steps use mainly the reconstruction theorem. Step (3) uses mainly the adelic characterization and Step (4) uses both. The residue theorem of rational functions is used in Step (5), while the last step requires only clever manipulations.

%\begin{enumerate}
    %\item 
\medskip

\noindent \textbf{(1)}. The first step is to establish an expression of the $J^K$-function in Corollary~\ref{c:3.2}. That is %that $\displaystyle {(1-q)^{-1}} J^{K}(0)$ has the following expression.
    \[
    \begin{split}
    \frac{1}{1-q}J^K(0) &= 1 + \frac{1-P}{5}\sum_{M\geq 1} Q^M \sum_{r \leq M} \sum_{\zeta^r=1} \Big( \frac{a_{M,\zeta}}{1-\zeta q} \Big)
    \\ &+ \frac{(1-P)^2}{5} \sum_{M\geq 1}Q^M \sum_{r \leq M} \sum_{\zeta^r=1} \Big( 
    \frac{b_{M,\zeta}}{1-\zeta q} + \frac{c_{M,\zeta}}{(1-\zeta q)^2}
    \Big)
    \\ & + \frac{(1-P)^3}{5} \sum_{M \geq 1} Q^M 
    \sum_{r \leq M} \sum_{\zeta^r =1} \left( \frac{\dd_{M,\zeta}}{1-q\zeta} + \frac{e_{M,\zeta}}{(1-\zeta q)^2} + \frac{f_{M,\zeta}}{(1-\zeta q)^3}
    \right),
    \end{split}
    \]
    where  $\sum_{\zeta^r=1}$ is the sum over \emph{primitive} $r$-th roots of unity. 
    We call $a_{M,\zeta}, c_{M,\zeta}$, $f_{M,\zeta}$ the coefficients of degree 0 terms, $b_{M,\zeta}$, $e_{M,\zeta}$ the coefficients of degree 1 terms, and $\dd_{M,\zeta}$ the degree 2 terms. This step uses mainly the reconstruction theorem.

\medskip
\noindent \textbf{(2)}.  %    \item[(2)]
    The "degree $0$ terms" $a_{M,\zeta}, c_{M,\zeta}$, $f_{M,\zeta}$ are the easiest to handle. Let $\zeta$ be a primitive $r$-th root of unity. In Theorem~\ref{Thm_deg0_inv}, we show that
    \[
    \begin{split}
    a_{M,\zeta} = 0,
   % \\
    c_{M,\zeta} = \frac{ \GV^{(3)}_{M/ r}  }{M^2},
    %\\
    f_{M,\zeta} =  \frac{-2 \GV^{(3)}_{M/ r}  }{M^3},
    \end{split}
    \]
    where 
    \[
    \GV_n^{(\gamma)} := \sum_{d|n} d^{\gamma} \GV_d.
    \]
    Note that the above invariants are defined to be zero if $M/r$ are \emph{not integers}. Two proofs will be given. The one uses virtual orbifold HRR formula is given in Subsection~\ref{subsection_degree0}. The one in Appendix~\ref{Appendix_degree0} uses only the reconstruction theorem.

\medskip
\noindent \textbf{(3)}.    %\item 
For higher degree terms, we show in Theorem~\ref{Thm_rpower} that 
    \[
    \begin{split}
     b_{rd, \zeta} &= \GV^{(1)}_d + \frac{1}{r^2}\Big( \frac{-\GV_d^{(3)}}{d^2} \Big), 
     \\
     \dd_{rd,\zeta} &= r d\GV^{(-1)}_d + \GV^{(1)}_d + \frac{1}{r^2} \Big( \frac{-\GV_d^{(3)}}{d^2} \Big)    +\frac{1}{r^3}\Big( \frac{-\GV_d^{(3)}}{d^3} \Big),   
     \\
     e_{rd,\zeta} &= \frac{1}{r^2}\Big( \frac{\GV_d^{(3)}}{d^2} \Big)  + \frac{1}{r^3} \Big( \frac{3\GV_d^{(3)}}{d^3} \Big).
    \end{split}
    \]
%    We put all computations of invariants in the Kawasaki's strata in Appendix. 
This step uses the virtual orbifold HRR formula and induction on Novikov variable.

%\medskip
%\noindent \textbf{(4)}. %    \item 
%In Proposition ~\ref{prop_gene_residue}, we show that ${(1-q)^{-1}} J^{\KK}(0)dq$ has residue only at roots of unity. We apply the residue theorem to produce recurrence relations for $b_d^0$ and $\dd_d^0$ and show in Theorem~\ref{t:3.13} that
%    \[
%    b^0_M = \GV_M^{(1)} , \quad \dd^1_{M} = M (\GV_M^{(-1)}).
%    \]

\medskip
\noindent \textbf{(4)}.   %  \item 
    Combining the above, we have 
    \[
    \begin{split}
    &\frac{1}{1-q}J^K(0) = 1+  \\
    & + \frac{(1-P)^2}{5} \sum_{r \geq 1} \sum_{\zeta^r=1 }  \sum_{d \geq 1} Q^{rd} \left( \frac{\GV_{d}^{(1)}-\frac{1}{r^2d^2}\GV^{(3)}_d}{1-\zeta q} +\frac{\frac{1}{r^2d^2}\GV^{(3)}_d}{(1-\zeta q)^2}\right)
    \\
    &+ \frac{(1-P)^3}{5}\sum_{r \geq 1} \sum_{\zeta^r=1 }  \sum_{d \geq 1} Q^{rd} 
    \left(\frac{rd \GV_d^{(-1)} +  \GV_d^{(1)}  - \frac{1}{r^2d^2} \GV_d^{(3)} - \frac{1}{r^3d^3} \GV_d^{(3)} }{1-\zeta q} \right.
    \\  &
    \hspace{5.5cm} \left. +
    \frac{\frac{\GV_d^{(3)}}{r^2d^2} + \frac{3\GV^{(3)}_d}{r^3d^3}}{(1-\zeta q)^2} + \frac{\frac{-2 \GV^{(3)}_d}{r^3d^3}}{(1-\zeta q)^3}
    \right).
    \end{split}
    \]
    A direct computation in Theorem~\ref{t:3.14} shows that this coincides with Conjecture \ref{conjectureJK}.
%\end{enumerate}

\subsection{Analytic properties of $J^K(0)$ and a degree argument} \label{s:3.2}
In this subsection, our main task is to prove Theorem~\ref{t:3.1} below. We first \emph{partially} expand $I^K$ in the basis $\{(1-P)^i
\}_{i=0,1, 2,3}$:
\begin{equation} \label{e:3.1}
    \begin{split}
        &\frac{1}{1-q}I^{\KK} \\
        = &\sum_{d\geq 0} Q^d \frac{\prod_{r=1}^{5d} (1-P^5 q^r) }{(\prod_{r=1}^d (1-Pq^r))^5}
        \\
        =  &\sum_{d \geq 0} Q^d \frac{ \prod_{r=1}^{5d} ( 1-q^r - q^r (P^5-1) ) }{ ( \prod_{r=1}^d (1-q^r - q^r(P-1)) )^5 }
        \\
        = &\sum_{d \geq 0} Q^d \frac{ \prod_{r=1}^{5d} (1-q^r)  }{ (\prod_{r=1}^d (1-q^r))^5 } \cdot \\
        & \cdot \Big( 1 + e_1(5d) \frac{1-P^5}{P^5} + e_2(5d) \cdot (\frac{1-P^5}{P^5})^2 + e_3(5d) (\frac{1-P^5}{P^5})^3 \Big) \cdot
        \\
        & \cdot \Big( 1 - 5 e_1(d) \frac{1-P}{P} + (15e_1(d)^2 -5e_2(d)) (\frac{1-P}{P})^2 - (35e_1(d)^3 \\
        & \qquad -30e_1(d)e_2(d))(\frac{1-P}{P})^3  \Big),
    \end{split}   
\end{equation}
where
\[
e_i(k) := \sum_{1\leq j_1 < \cdots < j_i\leq k} \frac{1}{1-q^{j_1}} \cdots \frac{1}{1-q^{j_i}}.
\]
For the third equality, the following equalities are used
%\begin{itemize}
 %   \item 
    \[
    \begin{split}
        \frac{1}{1-q^r - q^r(P-1)} &= \frac{1}{1-q^r} \sum_{i=0}^3 \Big(\frac{q^r}{1-q^r}\Big)^i (P-1)^i
        \\
        &= \frac{1}{1-q^r} \sum_{i=0}^3 \Big( \frac{1}{1-q^r} -1 \Big)^i (P-1)^i
        \\
        &= \frac{1}{1-q^r} \frac{1}{P}\sum_{i=0}^3  (-1)^i \Big(\frac{1}{(1-q^r)}\Big)^i \Big( \frac{1-P}{P}\Big)^i ,
    \end{split}
    \]
 %   \item 
    \[
    \begin{split}
        1-q^r - q^r(P^5-1) &= (1-q^r) \Big( 1 - \frac{q^r}{1-q^r} (P^5-1)\Big)
        \\
        &=(1-q^r) \Big( P^5 - \frac{P^5-1}{1-q^r}  \Big)
        \\
        &=P^5 (1-q^r) \Big( 1 + \Big(\frac{1}{1-q^r}\Big) \Big( \frac{1-P^5}{P^5}\Big) \Big) .
    \end{split}
    \]
%\end{itemize}
Note that in order to completely expand in terms of the basis, $1/P$ should have been rewritten as
\[
    \frac{1}{P} = \frac{1}{1-(1-P)} = 1 +(1-P) + (1-P)^2 + (1-P)^3.
\]
However, the precise form of expansion is not needed and \eqref{e:3.1} will suffice for our purpose below.

Fix a root of unity $\zeta$. Let $f \in \K= K^0(X)(q)[\![Q]\!] $. We will expand the $Q$-coefficients $f$ as a Laurent series in $(1- \zeta q)$, $K^0(X)[(1- \zeta q)^{-1}, (1- \zeta q)]\!]$, in terms of the basis $\{ (1- P)^i \}_{i=0}^3$ for $K^0(X)$. 
%Assume that a $K(X)$-valued function in $q$ (e.g., an element in $\mathcal{K}$) is expanded in the above basis $\{(1-P)^i \}_{i=0,1, 2,3}$ and in $(1-\zeta q)$.
\emph{Define} the \emph{degree $\deg_{q=\zeta^{-1}}$} as follows:
\[
\deg_{q=\zeta^{-1}} ( (1-\zeta q)^i (1-P)^j ) := i+j, \quad \deg_{q=\zeta^{-1}} (0) := \infty . 
\]
For an inhomogeneous element, the degree is defined to be the lowest degree of the homogeneous parts.
%$\deg$ of a polynomial = minimal degree of homogeneous parts
We then have
\begin{equation} \label{e:svaluation}
    \begin{split}
 \deg_{q=\zeta^{-1}} (a\cdot b) & \geq \deg_{q=\zeta^{-1}} (a) + \deg_{q=\zeta^{-1}} (b) \\
 \deg_{q=\zeta^{-1}} (a + b) & \geq \op{min} \{ \deg_{q=\zeta^{-1}} (a), \deg_{q=\zeta^{-1}} (b)\} .
\end{split}
\end{equation}
Indeed, one might wish to consider $\deg_{q=\zeta^{-1}}$ a \emph{discrete ``semi-valuation''}, and hence defines an adic topology, for these localized functions.
Expansion of $f\in \K$ in the ascending order of $\deg_{q=\zeta^{-1}}$, i.e., the $K$-valued Laurent series in $(1- \zeta q)$, is denoted as $(f)_{q=\zeta^{-1}}$ and is called the \emph{localization of $f$ at $q=\zeta^{-1}$}. For $f,g\in \K$, write 
\[
f \equiv_{q=\zeta^{-1}}^n g \quad \mbox{(resp.\ $f\equiv_{q=\zeta^{-1}} g$)},
\]
if $(f)_{q=\zeta^{-1}}$ and $(g)_{q=\zeta^{-1}}$ agree up to degree $\leq n$ (resp.\ they agree to all degrees).

\begin{remark} \label{r:3.1}
The semi-valuation above can be extended to the entire $\K$ by extending the totally ordered group from $\mathbb{Z}$ to $\mathbb{Z}^{\oplus 2}$ and define the degree of an element $f$ in $K^0(X)[(1- \zeta q)^{-1}, (1- \zeta q)]\!] [\![ Q ]\!]$ as
\[
  {\deg_{\zeta^{-1}}} \left( (1-\zeta q)^i (1-P)^j Q^d \right) = (d, i+j),
\]
with the total ordering on $\mathbb{Z}^{\oplus 2}$ being the lexicographic ordering. We will not use this but stick to the $\deg_{q=\zeta^{-1}}$.
\end{remark}

\begin{theorem} \label{t:3.1}
For any roots of unity $\zeta$, the $Q$-coefficients of $[J^{\KK}]_- $ are sums of terms with non-negative degrees with respect to $\deg_{q=\zeta^{-1}}$.
\end{theorem}
\begin{proof}
The reconstruction theorem \eqref{e:2.1} states that
\[
\begin{split}
 %&\qquad \frac{1}{1-q} 
 &J^{\KK}(0) = \\
 &%\frac{1}{1-q} 
 \sum_{d\geq 0} I^{\KK}_d Q^d \exp \left( \sum_{k>0}\frac{ \sum_{i=0}^3 \Psi^k(\e_i(Q)) (1-P^kq^{kd})^i}{k(1-q^k)} \right) \sum_{i=0}^3 r_i(q,Q) (1-Pq^d)^i.
\end{split}
\]
The right hand side of the formula has three parts, and we will show that each part has $\deg_{q=\zeta^{-1}} \geq 0$.
In the first part, we claim that $I^K_d$ have $\deg_{q=\zeta^{-1}} \geq 0$. This can be seen from the expansion of $I^K_d (q)$ in \eqref{e:3.1} and the paragraphs following it. In particular, the factor 
\begin{equation} \label{e:3.2}
     (1-q)^{-1} I^K_{0d}(q) = \frac{ \prod_{r=1}^{5d} (1-q^r)  }{\prod_{r=1}^d (1-q^r)^5 }
\end{equation}
is actually a polynomial in $q$ and $e_i(d)$ has poles of order no greater than $i$ at $q= \zeta^{-1}$.

On the other hand, the last part of formula, $\sum_{i=0}^3 r_i(q,Q) (1-Pq^d)^i$, consists of only the polynomial in $q$, as proven in Theorem~\ref{t:2.12}. Its $Q$-coefficients therefore have $\deg_{q=\zeta^{-1}} \geq 0$. 

In the middle $\exp$ factor, we first claim $\e_{0} (Q)=0$. Assuming the claim, the $\exp$ factor can be written as follows
\[
\begin{split}
    & \exp \left( \sum_{k>0}\frac{ \sum_{i=0}^3 \Psi^k(\e_i(Q)) (1-P^kq^{kd})^i}{k(1-q^k)} \right) \\
    =& \exp \left( \sum_{k>0}\frac{ \sum_{i=1}^3 \Psi^k(\e_i(Q)) (1-P^kq^{kd})^i}{k(1-q^k)} \right)
    \\
    = & \exp \left( \sum_{k>0}\frac{ \sum_{i=1}^3 \Psi^k(\e_i(Q)) (1-q^{kd}-q^{kd}(P^k-1))^i}{k(1-q^k)} \right).
\end{split}
\]
One see clearly that the exponent has $\deg_{q=\zeta^{-1}} \geq 0$. We now return to the proof of claim. By definition $\e_{00} =0$, and assume that $\e_{0M'} =0$ for all $M' <M$. By \eqref{e:3.2}, $I^K_{od}$ is a polynomial of $q$ and is divisible by $1-q$.
Following the reconstruction theorem and in particular \eqref{e:2.3}, we have
\[
 \e_{0M} = - \sum_{d_1+d_2 =M} I^K_{0 d_1} (1) r_{0 d_2}(1).
\]
Note that the LHS is a constant in $q$ while the RHS always contains a factor $I^K_{0 d_1} (q)$, a polynomial in $q$ divisible by $1-q$. We have therefore, $e_{0M} = \mathrm{RHS} (q=1) =0$.

The theorem follows by the semi-valuation property \eqref{e:svaluation}.
\end{proof}

Theorem~\ref{t:3.1} has the following important consequence on the analytic property of $\displaystyle {(1-q)^{-1}} J^{\KK}(0)$.

\begin{corollary} \label{c:3.2}
The small $J^K$-function has the following expansion in $Q$ and $q$
\[
    \begin{split}
    \frac{1}{1-q} J^K(0) &= 1 + \frac{1-P}{5}\sum_{M\geq 1} Q^M \sum_{r \leq M} \sum_{\zeta^r=1} \Big( \frac{a_{M,\zeta}}{1-\zeta q} \Big)
    \\ &+ \frac{(1-P)^2}{5} \sum_{M\geq 1}Q^M \sum_{r \leq M} \sum_{\zeta^r=1} \Big( 
    \frac{b_{M,\zeta}}{1-\zeta q} + \frac{c_{M,\zeta}}{(1-\zeta q)^2}
    \Big)
    \\ & + \frac{(1-P)^3}{5} \sum_{M \geq 1} Q^M 
    \sum_{r \leq M} \sum_{\zeta^r =1} \left( \frac{\dd_{M,\zeta}}{1-q\zeta} + \frac{e_{M,\zeta}}{(1-\zeta q)^2} + \frac{f_{M,\zeta}}{(1-\zeta q)^3}
    \right),
    \end{split}
\]
where $\displaystyle \sum_{\zeta^r=1}$ is the sum over all primitive $r$-th roots of unities.
\end{corollary}

\begin{proof}
By Remark~\ref{r:2.14}, the partial fraction expansion of
\[
 {(1-q)^{-1}} J^{\KK}(t, q, Q) \in \K = \K_+ \oplus K_-
\]
consists of the $\K_+$-projection, a Laurent polynomial in $q$, and the $\K_-$ projection, a rational function in $q$ with poles only at roots of unity and vanishing at $\infty$.
We know that ${(1-q)^{-1}} J^{\KK}(t=0)|_{\K_+} = 1$. Therefore, the RHS must have the specified form, provided that we can prove the coefficients of $Q^M$ have poles only at $r$-th roots of unity for $r \leq M$. This can be seen from the reconstruction theorem as follows.

%Consider the reconstruction theorem:
%\[
%\begin{split}
%  &\frac{1}{1-q} J^{\KK}(0) = \frac{1}{1-q} \cdot\\ & \cdot \sum_{d\geq 0} I^{\KK}_d Q^d \exp \left( \sum_{k>0}\frac{ \sum_{i=0}^3 \Psi^k(\e_i(Q)) (1-P^kq^{kd})^i}{k(1-q^k)} \right) \sum_{i=0}^3 r_i(q,Q) (1-Pq^d)^i.
%\end{split}
%\]

As in the proof of Theorem~\ref{t:3.1}, we employ the reconstruction theorem \eqref{e:2.1} and analyze the three factors $I^K$, the exponential, and the polynomial $r_i(q,Q)$ separately. 
In \eqref{e:3.1} of $I^K_d$, the factor $e_i(5d)$ in the coefficient of $Q^d$ can be cancelled by the numerator ${\prod_{r=1}^{5d} (1-q^r)}$ of the overall factor. The remaining denominator has poles only at $r$-th roots of unity with $r \leq d$.
%Note that any expressions with poles at roots of unity of order between $d+1$ and $5d$ will be cancelled by 
%\[
%\frac{\prod_{r=1}^{5d} (1-q^r)}{\prod_{r=1}^d (1-q^r)^5}.
%\]
For the exp factor, $\e_i(Q=0)=0$ for all $i$ and hence any expressions with poles at roots of unity of order $k$ appear with $Q^{jk}$ for $j \geq 1$.
Finally, the third factor $r_i$ are polynomial in $q$ and have no poles. This finishes the proof.
\end{proof}

In the above expression, we observe that the terms involving $a_{M,\zeta}, c_{M,\zeta}$, and $f_{M,\zeta}$ have $\deg_{q=\zeta^{-1}} =0$; those involving $b_{M,\zeta}$ and $e_{M,\zeta}$ have $\deg_{q=\zeta^{-1}} =1$; those involving $\dd_{M,\zeta}$ have $\deg_{q=\zeta^{-1}} =2$.

\subsection{The degree 0 terms} \label{subsection_degree0}
Let $\zeta$ be a primitive $r$-th roots of unity. In this subsection, we compute $a_{M,\zeta}, c_{M,\zeta}$, and $f_{M,\zeta}$.

\begin{theorem} \label{Thm_deg0_inv}
Let $d := M/r$. If $M/r$ are not integers, we have
\[
 a_{M,\zeta} = c_{M,\zeta} = f_{M,\zeta}=0.
\]
%All invariants are defined to be zero 
Otherwise,  we have
%All degree 0 terms are given as follows:
\[
%\begin{split}
    a_{M,\zeta} =0, \quad
    c_{M,\zeta} = \frac{ \GV_{d}^{(3)} }{r^2d^2}, \quad
    f_{M,\zeta} = \frac{-2 \GV_{d}^{(3)}}{r^3d^3}.
%\end{split}
\]
\end{theorem}

\begin{proof}
The following expansion will be used.
\begin{equation} \label{e:3.4}
\begin{split}
   \frac{1}{1-q L} &=  \frac{ 1 - (\zeta^{-1}L-1)+(\zeta^{-1}L-1)^2 }{1-\zeta q}  \\ 
    + &\frac{(\zeta^{-1}L-1) -2 (\zeta^{-1}L-1)^2}{(1-\zeta q)^2} + \frac{ (\zeta^{-1}L-1)^2 }{(1-\zeta q)^3} + O (\zeta^{-1}L -1)^3,
   \\
   \frac{1}{1-q} &=  \frac{1}{1-\zeta^{-1}} - \frac{\zeta^{-1}}{(1-\zeta^{-1})^2} (1-\zeta q) 
   \\
   & \quad + \frac{\zeta^{-2}}{(1-\zeta^{-1})^3} (1-\zeta q)^2 +  O(1-\zeta q)^3.
\end{split}
\end{equation}
%In the following computation, we will write $M=rd$. Now we have
Let $\Coeff(f(x);x^d)$ be the coefficient of $x^d$ in $f(x)$. By definition,
\[
\begin{split}
    a_{rd,\zeta} &= 5 \, \Coeff \Big(\frac{1}{1-q}J^K(0); Q^{rd} \frac{1-P}{1-\zeta q} \Big)
    \\
    &= \Coeff \Big( \frac{1}{1-q} Q^{rd}(1-P) \langle \frac{(1-P)^2 + (1-P)^3}{1-qL} \rangle^{X_{\zeta}}_{0,1,rd}; Q^{rd}\frac{1-P}{1-\zeta q} \Big)
    \\
    &= \Coeff \Big( \frac{1}{1-q} \langle \frac{(1-P)^2}{1-\zeta q} \rangle^{X_{\zeta}}_{0,1,rd}; \frac{1}{1-\zeta q} \Big)
    \\
    &= 0.
\end{split}
\]
In the last equality, Lemma~\ref{lemma_statement1} is used.
Similarly for $c_{rd,\zeta}$
\[
\begin{split}
    c_{rd,\zeta} &= 5 \, \Coeff \Big(\frac{1}{1-q}J^K(0); Q^{rd} \frac{(1-P)^2}{(1-\zeta q)^2} \Big)
    \\
    &= \Coeff \Big( \frac{1}{1-q} Q^{rd}(1-P)^2 \langle \frac{(1-P) + (1-P)^2}{1-qL} \rangle^{X_{\zeta}}_{0,1,rd}; Q^{rd}\frac{(1-P)^2}{(1-\zeta q)^2} \Big)
    \\
    &= \Coeff \Big( \frac{1}{1-q} \langle \frac{ (\zeta^{-1} L -1)(1-P)}{(1-\zeta q)^2} \rangle^{X_{\zeta}}_{0,1,rd}; \frac{1}{(1-\zeta q)^2} \Big)
    \\
    &= \frac{ \GV^{(3)}_{d}}{r^2d^2}.
\end{split}
\]
For the last equality, we use Lemma~\ref{lemma_statement2}. Note that the factor $(1-\zeta^{-1})$ will be cancelled by the constant term of $(1-q)^{-1}$. For $f_{rd,\zeta}$
\[
\begin{split}
    f_{rd,\zeta} &= 5 \, \Coeff \Big(\frac{1}{1-q}J^K(0); Q^{rd} \frac{(1-P)^3}{(1-\zeta q)^3} \Big)
    \\
    &= \Coeff \Big( \frac{1}{1-q} Q^{rd}(1-P)^3 \langle \frac{1 + (1-P) - (1-P)^3}{1-qL} \rangle^{X_{\zeta}}_{0,1,rd}; Q^{rd}\frac{(1-P)^3}{(1-\zeta q)^3} \Big)
    \\
    &= \Coeff \Big( \frac{1}{1-q} \langle \frac{ (\zeta^{-1} L -1)^2}{(1-\zeta q)^3} \rangle^{X_{\zeta}}_{0,1,rd}; \frac{1}{(1-\zeta q)^3} \Big)
    \\
    &= \frac{ -2 \GV_d^{(3)}}{r^3d^3}.
\end{split}
\]
Lemma~\ref{lemma_statement3} is used in the last equality. 

Lastly, when $M/r$ are not integers, we observe that 
\[
\langle (1-P)^2\rangle^{X_{\zeta}}_{0,1,M} = \langle (\zeta^{-1}L-1)(1-P)\rangle^{X_{\zeta}}_{0,1,M} =  \langle (\zeta^{-1}L-1)^2\rangle^{X_{\zeta}}_{0,1,M} =0
\]
by Lemma~\ref{lemma_statement1}, Lemma~\ref{lemma_statement2}, Lemma~\ref{lemma_statement3}, and the convention that $\GW_{M/r} =0$ if $M/r$ are not integers. Geometrically, this vanishing reflects the fact that the appearance of $r$-th roots of unity in $q$-expressions comes from the $mr$-fold covers (for $m \in \mathbb{Z}_{\geq 1}$) in the \emph{stem} contribution, where the degrees of curve classes are divisible by $mr$.
This completes the proof.
\end{proof}

\subsection{Higher degree terms}
Higher degree terms can be computed by induction on Novikov variable.

\begin{theorem} \label{Thm_rpower}
\[
\begin{split}
    &\frac{1}{1-q}J^{\KK}(0) = 1+ \frac{1}{5} \sum_{r \geq 1} \sum_{\zeta: {\rm ord}(\zeta)=r } (1-P)^2 \sum_{d \geq 1} Q^{rd} \left( \frac{ b_{rd,\zeta} }{1-\zeta q} +\frac{\frac{1}{r^2d^2}\GV^{(3)}_d}{(1-\zeta q)^2}\right)
    \\
    &+ \frac{1}{5}\sum_{r \geq 1} \sum_{\zeta: {\rm ord}(\zeta)=r } (1-P)^3 \sum_{d \geq 1} Q^{rd}\left(
    \frac{\dd_{rd,\zeta} }{1-\zeta q} + \frac{ e_{rd,\zeta} }{(1-\zeta q)^2} + \frac{\frac{-2 \GV^{(3)}_d}{r^3d^3}}{(1-\zeta q)^3}  
    \right),
\end{split}
\]
where 
 \[
    \begin{split}
     b_{rd, \zeta} &= \GV^{(1)}_d + \frac{1}{r^2}\Big( \frac{-\GV_d^{(3)}}{d^2} \Big)  ,
     \\
     e_{rd,\zeta} &= \frac{1}{r^2}\Big( \frac{\GV_d^{(3)}}{d^2} \Big)  + \frac{1}{r^3} \Big( \frac{3\GV_d^{(3)}}{d^3} \Big) ,
     \\
     \dd_{rd,\zeta} &= r d\GV^{(-1)}_d + \GV^{(1)}_d + \frac{1}{r^2} \Big( \frac{-\GV_d^{(3)}}{d^2} \Big)    +\frac{1}{r^3}\Big( \frac{-\GV_d^{(3)}}{d^3} \Big)   .
    \end{split}
\]
\end{theorem}

\begin{proof}
%We will use the following expansion:
%\[
%\begin{split}
%   \frac{1}{1-q L} = &\frac{ 1 - (\zeta^{-1}L-1) +(\zeta^{-1}L-1)^2 }{1-\zeta q} \\
%    & + \frac{(\zeta^{-1}L-1) -2 (\zeta^{-1}L-1)^2}{(1-\zeta q)^2} + \frac{ (\zeta^{-1}L-1)^2 }{(1-\zeta q)^3},
%   \\
%   \frac{1}{1-q} = &\frac{1}{1-\zeta^{-1}} - \frac{\zeta^{-1}}{(1-\zeta^{-1})^2} (1-\zeta q) \\
%   &+ \frac{\zeta^{-2}}{(1-\zeta^{-1})^3} (1-\zeta q)^2 + O(1-\zeta q)^3.
%\end{split}
%\]
It is a purely algebraic computation assuming Lemma \ref{lemma_statement1} to Lemma \ref{lemma_statement6}, where we use induction on Novikov variable. For the initial ($Q^1$) term, it involves only fake theory and can be computed using equation \ref{eqn:J-fake}.

The expansions in \eqref{e:3.4} imply that
\[
\begin{split}
b_{rd,\zeta} &= 5 \, \Coeff\Big( \frac{1}{1-q} J^{\KK}(0) ; Q^{rd} \frac{(1-P)^2}{1-\zeta q} \Big)
\\
&=  \Coeff \Big( \frac{1}{1-q} Q^{rd} (1-P)^2 \langle \frac{(1-P)+(1-P)^2}{1-q L}\rangle^{X_{\zeta}}_{0,1,rd} ; Q^{rd} \frac{(1-P)^2}{1-\zeta q}\Big)
\\
&= \Coeff \Big(  \frac{1}{1-q} \langle \frac{(1-P) + (1-P)^2 - (1-P)(\zeta^{-1} L-1)}{1-\zeta q}\rangle^{X_{\zeta}}_{0,1,rd}   ; \frac{1}{1-\zeta q} \Big)
\\
&\qquad + \Coeff \Big( \frac{1}{1-q} \langle \frac{(1-P)(\zeta^{-1} L-1)}{(1-\zeta q)^2}
\rangle^{X_{\zeta}}_{0,1,rd} ; \frac{1}{1-\zeta q}\Big) 
\\
& = \frac{1}{1-\zeta^{-1}}\langle (1-P) + (1-P)^2 - (1-P)(\zeta^{-1} L-1)\rangle^{X_{\zeta}}_{0,1,rd} 
\\
&\qquad - \frac{\zeta^{-1}}{(1-\zeta^{-1})^2} \langle (1-P)(\zeta^{-1} L-1) \rangle^{X_{\zeta}}_{0,1,rd}
\\
& = \GV^{(1)}_d + \frac{1}{r^2} \Big( \frac{-\GV_d^{(3)}}{d^2} \Big).
\end{split}
\]
In the last equality, Lemmas \ref{lemma_statement1}, \ref{lemma_statement2}, \ref{lemma_statement3}, and Lemma \ref{lemma_statement5} are used. 

For $e_{rd,\zeta}$, we have
\[
\begin{split}
 & e_{rd,\zeta} 
= 5 \, \Coeff\Big( \frac{1}{1-q} J^{\KK}(0) ; Q^{rd} \frac{(1-P)^3}{(1-\zeta q)^2} \Big)
\\
=  &\Coeff \Big( \frac{Q^{rd} (1-P)^3}{1-q}   \langle \frac{1+(1-P)-(1-P)^3}{1-q L}\rangle^{X_{\zeta}}_{0,1,rd} ; Q^{rd} \frac{(1-P)^3}{(1-\zeta q)^2}\Big)
\\
= &\Coeff \Big(  %\frac{1}{1-q} 
\langle \frac{ (\zeta^{-1} L-1) + (1-P)(\zeta^{-1} L-1) -2(\zeta^{-1} L-1)^2 }{(1-q)(1-\zeta q)^2}\rangle^{X_{\zeta}}_{0,1,rd}   ; \frac{1}{(1-\zeta q)^2} \Big)
\\
&\qquad + \Coeff \Big( %\frac{1}{1-q} 
\langle \frac{(\zeta^{-1} L-1)^2}{(1-q)(1-\zeta q)^3}
\rangle^{X_{\zeta}}_{0,1,rd} ; \frac{1}{(1-\zeta q)^2}\Big) 
\\
 = &\frac{1}{1-\zeta^{-1}}\langle (\zeta^{-1} L-1) + (1-P)(\zeta^{-1} L-1) -2(\zeta^{-1} L-1)^2 \rangle^{X_{\zeta}}_{0,1,rd} 
\\
&\qquad - \frac{\zeta^{-1}}{(1-\zeta^{-1})^2} \langle (\zeta^{-1} L-1)^2 \rangle^{X_{\zeta}}_{0,1,rd}
\\
= & \frac{1}{r^2}\Big( \frac{\GV_d^{(3)}}{d^2} \Big)  + \frac{1}{r^3} \Big( \frac{3\GV_d^{(3)}}{d^3} \Big).
\end{split}
\]
In the last equality, Lemmas \ref{lemma_statement2}, \ref{lemma_statement3}, and \ref{lemma_statement4} are used.

Finally, for $\dd_{rd,\zeta}$
%{\small
\[
\begin{split}
& \quad \dd_{rd,\zeta} 
= 5 \, \Coeff\Big( \frac{1}{1-q} J^{\KK}(0) ; Q^{rd} \frac{(1-P)^3}{1-\zeta q} \Big)
\\
&=  \Coeff \Big( \frac{1}{1-q} Q^{rd} (1-P)^3 \langle \frac{1+(1-P)-(1-P)^3}{1-q L}\rangle^{X_{\zeta}}_{0,1,rd} ; Q^{rd} \frac{(1-P)^3}{1-\zeta q}\Big)
\\
&= \Coeff \Big(  %\frac{1}{1-q} 
 \langle \frac{[1 -(\zeta^{-1} L-1) +(\zeta^{-1} L-1)^2 ] +(1-P) [1 -(\zeta^{-1} L-1)] }{(1-q)(1-\zeta q)}\rangle^{X_{\zeta}}_{0,1,rd}   ; \frac{1}{1-\zeta q} \Big)
\\
&+ \Coeff \Big( %\frac{1}{1-q} 
 \langle \frac{(\zeta^{-1} L-1)+(1-P)(\zeta^{-1} L-1)-2(\zeta^{-1} L-1)^2}{(1-q)(1-\zeta q)^2} \\
&\qquad \qquad + \frac{(\zeta^{-1} L-1)^2}{(1-q)(1-\zeta q)^3} \rangle^{X_{\zeta}}_{0,1,rd} ; \frac{1}{1-\zeta q}\Big) 
\\
& = \frac{1}{1-\zeta^{-1}}\langle 1+(1-P)-(\zeta^{-1} L-1)-(1-P)(\zeta^{-1} L-1)+(\zeta^{-1} L-1)^2 \rangle^{X_{\zeta}}_{0,1,rd}
\\
& - \frac{\zeta^{-1}}{(1-\zeta^{-1})^2} \langle (\zeta^{-1} L-1)+(1-P)(\zeta^{-1} L-1)-2(\zeta^{-1} L-1)^2
 \rangle^{X_{\zeta}}_{0,1,rd} \\
 &- \frac{\zeta^{-2}}{(1-\zeta^{-1})^3} \langle (\zeta^{-1} L-1)^2
 \rangle^{X_{\zeta}}_{0,1,rd}
\\
& = r d\GV^{(-1)}_d + \GV^{(1)}_d + \frac{1}{r^2} \Big( \frac{-\GV_d^{(3)}}{d^2} \Big)    +\frac{1}{r^3}\Big( \frac{-\GV_d^{(3)}}{d^3} \Big).
\end{split}
\]
%}
For the last equality, Lemmas \ref{lemma_statement2}, \ref{lemma_statement3}, \ref{lemma_statement4}, \ref{lemma_statement5}, and \ref{lemma_statement6} are used. 
\end{proof}

\subsection{Conclusion of the proof}
We have shown that 
\begin{equation} \label{eqn:Conj}
    \begin{split}
    &\frac{1}{1-q}J^K(0) \\
    = & 1+ \frac{(1-P)^2}{5} \sum_{r \geq 1} \sum_{\zeta: {\rm ord}(\zeta)=r }  \sum_{d \geq 1} Q^{rd} \left( \frac{\GV_{d}^{(1)}-\frac{1}{r^2d^2}\GV^{(3)}_d}{1-\zeta q} +\frac{\frac{1}{r^2d^2}\GV^{(3)}_d}{(1-\zeta q)^2}\right)
    \\
    &+ \frac{(1-P)^3}{5}\sum_{r \geq 1} \sum_{\zeta: {\rm ord}(\zeta)=r }  \sum_{d \geq 1} Q^{rd} \cdot \\
   & \cdot \left(\frac{rd \GV_d^{(-1)} +  \GV_d^{(1)}  - \frac{\GV_d^{(3)}}{r^2d^2}  - \frac{\GV_d^{(3)}}{r^3d^3}  }{1-\zeta q} %\right.
%    &\hspace{5cm} \left.
   + \frac{\frac{\GV_d^{(3)}}{r^2d^2} + \frac{3\GV^{(3)}_d}{r^3d^3}}{(1-\zeta q)^2} + \frac{\frac{-2 \GV^{(3)}_d}{r^3d^3}}{(1-\zeta q)^3}
    \right) .
    \end{split}
\end{equation}

%We are ready to prove the Conjecture \ref{conjectureJK}.

\begin{theorem} \label{t:3.14}
Equation \ref{eqn:Conj} is equivalent of Conjecture \ref{conjectureJK}.
\end{theorem}

\begin{proof}
We use the following expansions
\[
\begin{split}
    \frac{1}{1-q^M} &\equiv_{q=\zeta^{-1}} \frac{1}{ M(1-\zeta q) - \binom{M}{2} (1-\zeta q)^2 + \binom{M}{3}(1-\zeta q)^3 -\dots   }
    \\
    &= \frac{1}{M(1-\zeta q)} + \frac{M-1}{2M} + \frac{M^2-1}{12M}(1-\zeta q) + O((1-\zeta q)^2),
    \\
    \frac{1}{(1-q^M)^2} &\equiv_{q=\zeta^{-1}} \frac{1}{M^2(1-\zeta q)^2} + \frac{M-1}{M^2} \frac{1}{1-\zeta q} + \mbox{ regular terms},
    \\
    \frac{1}{(1-q^M)^3} &\equiv_{q=\zeta^{-1}} \frac{1}{M^3 (1-\zeta q)^3} + \frac{3(M-1)}{2M^3(1-\zeta q)^2} + \frac{2M^2-3M+1}{2M^3(1-\zeta q)} \\
    & \qquad + \mbox{ regular terms},
\end{split}
\]
where $\zeta$ is any (not necessary primitive) $M$-th roots of unity.

Start with the coefficient of $(1-P)^2 Q^M$ in the expression of Conjecture \ref{conjectureJK}. Let $\zeta$ be any primitive $r$-th roots of unity with $r|M$.

\[
\begin{split}
    &5 \, \Coeff \Big( \frac{1}{1-q} J^{\KK}(0) ; (1-P)^2 Q^M \Big)
    \\
    &= \sum_{k|M} \left(\frac{ \frac{M}{k}(k-1) }{1-q^k} + \frac{\frac{M}{k}}{(1-q^k)^2}\right) \GV_{M/k} 
    \\
    &\equiv_{q=\zeta^{-1}} \sum_{k|\frac{M}{r}} \left( \frac{ \frac{M}{rk}(rk-1) }{1-q^{rk}} + \frac{\frac{M}{rk}}{(1-q^{rk})^2}\right) \GV_{M/rk} + \mbox{ regular terms}
    \\
    &\equiv_{q=\zeta^{-1}} \sum_{k|\frac{M}{r}} \left( \frac{ \frac{M}{rk}(rk-1) }{rk(1-\zeta q)} + \frac{ \frac{M}{rk}(rk-1) }{(rk)^2(1-\zeta q)} + \frac{\frac{M}{rk}}{((rk)^2(1-\zeta q)^2}\right) \GV_{M/rk} \\
    & \qquad + \mbox{ regular terms}
    \\
    & \equiv_{q=\zeta^{-1}}  \left(  \frac{  \GV_{M/r}^{(1)} - \frac{1}{M^2} \GV_{M/r}^{(3)}   }{1-\zeta q} +\frac{ \frac{1}{M^2} \GV^{(3)}_{M/r}  }{(1-\zeta q)^2} \right) + \mbox{ regular terms}.
\end{split}
\]
This coincides with the equation $\ref{eqn:Conj}$ by taking $M=rd$.

Similarly for the coefficient of $(1-P)^3Q^M$
\[
\begin{split}
    &5 \, \Coeff \Big( \frac{1}{1-q} J^{\KK}(0) ; (1-P)^3 Q^M \Big)
    \\
    &= \sum_{k|M} \left(\frac{ M+k^2 - \frac{M}{k}-1 }{1-q^k} + \frac{\frac{M}{k}+3}{(1-q^k)^2} - \frac{2}{(1-q^k)^3}\right) \GV_{M/k} 
    \\
    &\equiv_{q=\zeta^{-1}} \sum_{k|\frac{M}{r}}\left(\frac{ M+(rk)^2 - \frac{M}{rk}-1 }{1-q^{rk}} + \frac{\frac{M}{rk}+3}{(1-q^{rk})^2} - \frac{2}{(1-q^{rk})^3}\right) \GV_{M/rk} \\
    & \qquad + \mbox{ regular terms}
    \\
    &\equiv_{q=\zeta^{-1}}  \sum_{k|\frac{M}{r}} \left( \frac{\GV_{M/r}}{ (1-\zeta q) } \Big( 
    (rk) +\frac{M}{rk}  + \frac{-M-1}{(rk)^3} \right.
    \Big) 
    \\
    & \qquad \left. + \frac{\GV_{M/r}}{(1-\zeta q)^2} \Big(   \frac{M+3}{(rk)^3}   \Big) + \frac{\GV_{M/r}}{ (1-\zeta q)^3 } \Big(    \frac{-2}{(rk)^3}  \Big) \right) + \mbox{ regular terms}
    \\
    & \equiv_{q=\zeta^{-1}} \sum_{k|\frac{M}{r}} \left(\frac{ M \GV_{M/r}^{(-1)} +  \GV_{M/r}^{(1)}  - \frac{1}{M^2} \GV_{M/r}^{(3)} - \frac{1}{M^3} \GV_{M/r}^{(3)}  }{1-\zeta q}  \right.
    \\
    &\qquad \left. + \frac{ \frac{\GV_{M/r}^{(3)}}{M^2} + \frac{3\GV^{(3)}_{M/r}}{M^3} }{(1-\zeta q)^2}  + \frac{\frac{-2 \GV^{(3)}_{M/r}}{M^3}}{(1-\zeta q)^3} \right) + \mbox{ regular terms}.
\end{split}
\]
This coincides with the equation \ref{eqn:Conj} by taking $M=rd$. The proof is now complete.

\end{proof}

%\section{Proof of identities}
\section{Some computations of contributions from Kawasaki's strata} \label{appen_Kawasaki}

The main purpose of this section is to prove a few identities needed in the proofs in Section~\ref{section_proof}.

We start with the fake theory:
\begin{equation} \label{eqn:J-fake}
\begin{split}
    \frac{1}{1-q} J^{\fake} (0) = 1 &+ \frac{1}{5}(1-P)^2\sum_d Q^d\frac{d\GW_d}{(1-q)^2}
    \\
    & +\frac{1}{5}(1-P)^3 \sum_d Q^d \left( \frac{(3+d) \GW_d}{(1-q)^2} - \frac{2\GW_d}{(1-q)^3}  \right),
\end{split}
\end{equation}
see Theorem~\ref{thm_coh_fake} and reference therein for more detail.

For the stem theory, let $\zeta$ be a primitive $r$-th roots of unity.
We prove the following statements about the invariants on Kawasaki's strata with $r\geq 2$ in the following order. Same formulas hold true when $r=\zeta=1$.
%\begin{enumerate}
%    \item 
    \[
    \langle (1-P)^2 \rangle^{X_{\zeta}}_{0,1,M} =0.
    \]
%    \item 
    \[
    \langle (\zeta^{-1} L-1)(1-P) \rangle^{X_{\zeta}}_{0,1,M} = (1-\zeta^{-1}) \Big( \frac{\GV^{(3)}_{M/r}}{M^2} \Big).
    \]
%    \item 
    \[
    \langle (\zeta^{-1} L-1)^2 \rangle^{X_{\zeta}}_{0,1,M} = (1-\zeta^{-1})\Big( \frac{-2\GV^{(3)}_{M/r}}{M^3} \Big).
    \]
%    \item
    \[
    \langle \zeta^{-1} L-1 \rangle^{X_{\zeta}}_{0,1,M} =  (1-\zeta^{-1})\Big(\frac{-\GV^{(3)}_{M/r}}{M^3}\Big) - \zeta^{-1} \Big(\frac{2\GV^{(3)}_{M/r}}{M^3} \Big).    
    \]
%    \item 
    \[
    \langle 1-P \rangle^{X_{\zeta}}_{0,1,M} = (1-\zeta^{-1}) \GV^{(1)}_{M/r} + \zeta^{-1} \Big( \frac{ \GV^{(3)}_{M/r} }{M^2} \Big)  .
    \]
%    \item 
    \[
    \langle 1 \rangle^{X_{\zeta}}_{0,1,M} =  (1-\zeta^{-1}) M\GV^{(-1)}_{M/r}  + \zeta^{-1} \Big(\frac{\GV^{(3)}_{M/r}}{M^3}\Big).
    \]
Here $\GV^{(\gamma)}_d := \sum_{k|d} k^{\gamma} \GV_k $.
If $(M/r)$ are not integers, $\GV^{(\gamma)}_{M/r}$ are set to 0 by convention.

Recall the notation in \eqref{e:kawasaki}
\[
\begin{split}
    & \Big[ T_1(L), T(L),\dots, T(L), T_{n+2}(L) \Big]^{X_{\zeta}}_{0,n+2,d}
    \\
    := & \int_{[\M_{0,n+2,d}^X(\zeta)]^{\vir}} \td( T_{\M} ) \ch \left( 
    \frac{ \ev_1^* (T_1(L)) \ev_{n+2}^* (T_{n+2}(L)) \prod_{i=2}^{n+1} \ev_i^* T(L) }
    {\Tr (\Lambda^* N^*_{\M})}  \right).
\end{split}
\]
This can be interpreted as \emph{twisted} (cohomological) GW invariants with the twisting class
\[
\td (T_{\M}) \ch \Big( \frac{1}{\Tr (\Lambda^* N^*_{\M})} \Big).
\]
Such twisting classes come from the deformation theory of the moduli of stable maps and consist of three parts. See e.g., \cite[Section~8]{Givental_Tonita_2011}.
\begin{enumerate}
    \item[type $A$.] $\displaystyle \td(\pi_*^K \ev^*(TX))\prod_{k=1}^{r-1} \td_{\zeta^k} (\pi_*^K \ev^*(T_X \otimes \mathbb{C}_{\zeta^k})) $, where $\pi: \mathcal{C} \rightarrow \M$ and $\ev: \mathcal{C} \rightarrow X/ \mathbb{Z}_r$ form the universal (orbifold) stable map diagram. $\mathbb{C}_{\zeta^k}$ is the line bundle over $B\mathbb{Z}_r$ with $g$ acts by $\zeta^k$, and for any line bundle $l$, define the invertible multiplicative characteristic classes
    \[
    \td(l) := \frac{c_1(l)}{1-e^{-c_1(l)}}, \quad \td_{\lambda}(l) := \frac{1}{1-\lambda e^{-c_1(l) }}
    \]
    \item[type $B$.] $\displaystyle \td(\pi_*^K(-L^{-1})) \prod_{k=1}^{r-1} \td_{\zeta^k} (\pi_*^K(-L^{-1} \otimes \ev^* (\mathbb{C}_{\zeta^k}) ))$, where $L=L_{n+3}$ is the universal cotangent line bundle of 
    \[
    \mathcal{C} \cong \M_{0,n+3}^{X/\mathbb{Z}_r ,d} (g, 1,\dots,1,g^{-1},1).
    \]
    \item[type $C$.] $\displaystyle \td^{\vee}(-\pi_*^K i_* \mathcal{O}_{Z_g}) \td^{\vee}(-\pi_*^K i_*\mathcal{O}_{Z_1}) \prod_{i=1}^{k-1}\td^{\vee}_{\zeta^k}(-\pi_*^K i_*\mathcal{O}_{Z_1})$, where $Z_1$ stands for unramified nodal locus, and $Z_g$ stands for ramified one with $i: Z \rightarrow \mathcal{C}$ the embedding of nodal locus. For any line bundle $l$,
    \[
     \td^{\vee}(l) = \frac{-c_1(l)}{1-e^{c_1(l)}}, \quad \td^{\vee}_{\lambda}(l) = \frac{1}{1-\lambda e^{c_1(l) }}.
    \]
\end{enumerate}
We start with the following two observations.
\begin{lemma} \label{l:4.1}
\[
\td (T_{\M}) \ch \Big( \frac{1}{\Tr (\Lambda^* N^*_{\M})} \Big) =: r^{-n} + T_{0,n+2,d}(\zeta),
\]
where $T_{0,n+2,d}(\zeta)\in H^{>0}(\M^X_{0,n+2,d}(\zeta))$, with $0$, $n$, $d$, and $\zeta$ inherited from $\M = \M^X_{0,n+2,d}(\zeta)$.
\end{lemma}
\begin{proof}
Note that $N_{\M}$ has virtual dimension $(r-1)n$ since $\M$ is considered as Kawasaki strata in $\M^X_{0,nr+2,d}$. 

Only the twisting classes on normal bundle give nontrivial constant. It is of the following form:
\[
\prod_{k=1}^{r-1} \td_{\zeta^k} \Big( \pi^K_*(  E \otimes \ev^* (\mathbb{C}_{\zeta^k}) )\Big),
\]
where $E$ is a virtual bundle of rank $n$. The constant term is given by
\[
\Big( \prod_{k=1}^{r-1} \frac{1}{1-\zeta^k} \Big)^n = r^{-n}.
\]

\end{proof}
\begin{lemma} \label{lemma_vanishingH^2}
For any Calabi-Yau threefolds $X$, if $\deg_{\mathbb{C}} \phi_1 \geq 2$ then
\[
\langle \tau_{k_1}(\phi_1),\dots, \tau_{k_n}(\phi_n) \rangle_{g,n,\beta\neq 0}^{\tw} =0,
\]
where $\tw$ denotes cohomological GW invariants with twistings by any combinations of the three types defined above.
\end{lemma}

\begin{proof}
%This well-known fact has many different proofs. All proofs use the fact that the virtual dimension of $\M_{g,n}(X, \beta)$ is equal to $n$.
%
Let 
\[
 \pi_1 : \M_{g,n}(X, \beta) \to \M_{g,1} (X, \beta)
\]
be the forgetful map forgetting the last $n-1$ marked points and $T\in H^*(\M_{g,n}(X,\beta))$ be the twisting class.
By projection formula
\[
\begin{split}
\int_{[\M_{g,n}(X, \beta)]^{\vir}} T \ &\prod_{i=1}^{n} \Big(\psi_i^{k_i} \ev_i^*\phi_i\Big) 
\\
&= \int_{[\M_{g,1}(X, \beta)]^{\vir}} (\ev_1^* \phi_1 ) \ (\pi_1)_*\left( T \ \psi_1^{k_1}\prod_{i=2}^{n} \psi_i^{k_i} \ev_i^*\phi_i \right) .
\end{split}
\] 
Since $[\M_{g,1}(X,\beta)]^{\vir}$ has virtual dimension 1, while $\deg_{\mathbb{C}}(\phi_i) \geq 2$, the last equation must vanish.
\end{proof}

\begin{lemma} \label{lemma_statement1}
\[
    \langle (1-P)^2 \rangle^{X_{\zeta}}_{0,1,M} =0.
\]
\end{lemma}
\begin{proof}
It is a consequence of Lemma~\ref{lemma_vanishingH^2} since $ \deg_{\mathbb{C}} \ch (1-P)^2 \geq 2$.
\end{proof}

The following expression of leg and tail contributions will be used.
\begin{equation} \label{eqn_leg_tail}
    \begin{split}
{\rm leg}_{\zeta}(q) &= \Psi^r \Big(  [J^{\KK}_{q=1}]_+ - (1-q)  \Big)
\\
&= \Psi^r \Big(  (1-P)^2 (\cdots) +(1-P)^3(\cdots)\Big)
\\
\delta_{\zeta}(q) &= (1-\zeta^{-1}q) + \tail_{\zeta}(\zeta^{-1}q)
\\
&= (1-\zeta^{-1}q) + (1-P)^2(\cdots) + (1-P)^3 (\cdots).
\end{split}
\end{equation}
%Proposition \ref{Prop_Stem_inv} with input $t=0$ gives:
%\[
%\begin{split}
%    \sum_d Q^d \langle (1-P)^2 \rangle^{X_{\zeta}}_{0,1,d}
%    =
%    \sum_{n,d} \frac{Q^{rd}}{n!}  \Big[ (1-P)^2 , \leg_{\zeta}(L), \dots, \leg_{\zeta}(L), \delta_{\zeta}(L^{1/r}) \Big]^{X_{\zeta}}_{0,2+n.d}.
%\end{split}
%\]
%We recall the following:
%\[
%\begin{split}
%{\rm leg}_{\zeta}(q) &= \Psi^r \Big(  [J^{\KK}_{q=1}]_+ - (1-q)  \Big)
%\\
%&= \Psi^r \Big(  (1-P) (\cdots) + (1-P)^2 (\cdots) +(1-P)^3(\cdots)\Big)
%\\
%\delta_{\zeta}(q) &= (1-\zeta^{-1}q) + \tail_{\zeta}(\zeta^{-1}q)
%\\
%&= (1-\zeta^{-1}q) + (1-P)(\cdots) + (1-P)^2(\cdots) + (1-P)^3 \cdots.
%\end{split}
%\]
%The invariant vanishes unless the dimension of moduli space and the degree of classes match. We have:
%\[
%\begin{split}
% &\sum_{d} Q^d \Big[ (1-P)^2,1-\zeta^{-1}L^{1/r} \Big]^{X_{\zeta}}_{0,2,d}  \\
%= &(1-\zeta^{-1})\sum_{d} Q^{rd} \int_{ [\M_{0,2,d}^X(\zeta)]^{\vir} } \ch\Big( \ev_1^* ( (1-P)^2 ) \Big)
%=0.
%\end{split}
%\]
%The integration equals to zero by the fundamental %class axiom.

\begin{lemma} \label{lemma_statement2}
 \[
    \langle (\zeta^{-1} L-1)(1-P) \rangle^{X_{\zeta}}_{0,1,M} = (1-\zeta^{-1}) \Big( \frac{\GV^{(3)}_{M/r}}{M^2} \Big).
\]
\end{lemma}
\begin{proof}
Proposition \ref{Prop_Stem_inv} with input $t=0$ gives:
\[
\begin{split}
    &\sum_{d} Q^d\langle (\zeta^{-1} L-1)(1-P) \rangle^{X_{\zeta}}_{0,1,d} 
    \\ 
    & \hspace{1cm} =\sum_{n,d} \frac{Q^{rd}}{n!}  \Big[ (L^{1/r}-1)(1-P) , \leg_{\zeta}(L), \dots, \leg_{\zeta}(L), \delta_{\zeta}(L^{1/r}) \Big]^{X_{\zeta}}_{0,2+n.d}.
\end{split}
\]
$\leg_{\zeta}(q)$ and $\delta_{\zeta}(q)$ can be expanded as in equation (\ref{eqn_leg_tail}).
The invariant vanishes unless the dimension of moduli space and the degree of classes match. We have:
\[
\begin{split}
  & \sum_d Q^d\langle (\zeta^{-1} L-1)(1-P) \rangle^{X_{\zeta}}_{0,1,d} \\
 =  & \sum_{d} Q^{rd} \Big[ (L^{1/r}-1)(1-P),1-\zeta^{-1}L^{1/r} \Big]^{X_{\zeta}}_{0,2,d} 
\\
=  & (1-\zeta^{-1}) \sum_{d} Q^{rd} \int_{ [\M_{0,2,d}^X(\zeta)]^{\vir} } \ch\Big( (L_1^{1/r}-1) \ev_1^*(1-P) ) \Big)
\\
= & \frac{(1-\zeta^{-1})}{r} \sum_{d} Q^{rd} \int_{ [\M_{0,2,d}^X(\zeta)]^{\vir} } c_1(L_1) \ev_1^* (H)
\\
=  & \frac{(1-\zeta^{-1})}{r^2} \sum_d Q^{rd} (d\GW_d) = \frac{(1-\zeta^{-1})}{r^2}\sum_d Q^{rd} \Big( \frac{\GV^{(3)}_d}{d^2} \Big) .
\end{split}
\]
The extra power of $r$ in the denominator comes from the virtual fundamental class. 
Compare the coefficient of $Q^M$ gives the result.
\end{proof}

\begin{lemma} \label{lemma_statement3}
    \[
    \langle (\zeta^{-1} L-1)^2 \rangle^{X_{\zeta}}_{0,1,M} = (1-\zeta^{-1})\Big( \frac{-2\GV^{(3)}_{M/r}}{M^3} \Big).
    \]
\end{lemma}
\begin{proof}
The computation is similar to Lemma \ref{lemma_statement2}.
\[
\begin{split}
&\sum_{d} Q^d \langle (\zeta^{-1} L-1)^2 \rangle^{X_{\zeta}}_{0,1,d} \\
& = \sum_{d} Q^{rd} \Big[ (L^{1/r}-1)^2,1-\zeta^{-1}L^{1/r} \Big]^{X_{\zeta}}_{0,2,d} 
\\
& = (1-\zeta^{-1}) \sum_{d} Q^{rd} \int_{ [\M_{0,2,d}^X(\zeta)]^{\vir} } \ch\Big( (L_1^{1/r}-1)^2 \Big)
\\
&= \frac{(1-\zeta^{-1})}{r^2} \sum_{d} Q^{rd} \int_{ [\M_{0,2,d}^X(\zeta)]^{\vir} } c_1(L_1)^2
\\
&= \frac{(1-\zeta^{-1})}{r^3} \sum_d Q^{rd} (-2\GW_d) = \frac{(1-\zeta^{-1})}{r^3} \sum_d Q^{rd} \Big( \frac{-2 \GV^{(3)}_d}{d^3} \Big).
\end{split}
\]
Comparing the coefficient of $Q^M$ gives the result.
\end{proof}

\begin{lemma}\label{lemma_statement4}
    \[
    \langle \zeta^{-1} L-1 \rangle^{X_{\zeta}}_{0,1,M} =  (1-\zeta^{-1})\Big(\frac{-\GV^{(3)}_{M/r}}{M^3}\Big) - \zeta^{-1} \Big(\frac{2\GV^{(3)}_{M/r}}{M^3} \Big).    
    \]
\end{lemma}
\begin{proof}
Proposition \ref{Prop_Stem_inv} with input $t=0$ gives:
\[
\begin{split}
& \sum_{d} Q^d \langle \zeta^{-1} L-1 \rangle^{X_{\zeta}}_{0,1,d} 
\\
= & \sum_{n,d} \frac{Q^{rd}}{n!}  \Big[ L^{1/r}-1 , \leg_{\zeta}(L), \dots, \leg_{\zeta}(L), \delta_{\zeta}(L^{1/r}) \Big]^{X_{\zeta}}_{0,2+n,d}
\\
= & \sum_{d} Q^{rd} \Big[ L^{1/r}-1,1-\zeta^{-1}L^{1/r} \Big]^{X_{\zeta}}_{0,2,d} + \Big[ L^{1/r}-1, \leg_{\zeta}(L) ,1-\zeta^{-1}L^{1/r} \Big]^{X_{\zeta}}_{0,3,0}
\\
= & \sum_{d} Q^{rd} \int_{ [\M_{0,2,d}^X(\zeta)]^{\vir} }  \Big( \frac{c_1(L_1)}{r} +\frac{c_1(L_1)^2}{2r^2} \Big)  \cdot \Big( (1-\zeta^{-1}) - \zeta^{-1} \frac{c_1(L_2)}{r} \Big) 
\\
= & \sum_{d} Q^{rd} \Big[ (1-\zeta^{-1})\Big(\frac{-\GV^{(3)}_d}{r^3d^3}\Big) - \zeta^{-1} \Big(\frac{2\GV^{(3)}_d}{r^3d^3} \Big)    \Big].
\end{split}
\]
In the third equality, the computation reduce to 2-pointed and untwisted case. We use the following observations. The twisting class of types A gives no contribution since $X$ is Calabi-Yau threefold. The twisting class of types B gives no contribution since $L$ is trivial on $\M^X_{0,2,d}(\zeta)$. 
Finally, the twisting class of type C and the term with $\leg_{\zeta}(L)$ give no contribution when $d\neq 0$ by Lemma \ref{lemma_vanishingH^2}. When $d=0$, the invariant reduces to the Poincar\'e pairing on $X$ and equals zero. 

Compare the coefficient of $Q^M$ gives the result.

\end{proof}

\begin{lemma}\label{lemma_statement5}
    \[
    \langle 1-P \rangle^{X_{\zeta}}_{0,1,M} = (1-\zeta^{-1}) \GV^{(1)}_{M/r} + \zeta^{-1} \Big( \frac{ \GV^{(3)}_{M/r} }{M^2} \Big)  .
    \]
\end{lemma}
\begin{proof}
The computation is similar to Lemma \ref{lemma_statement4} with nonzero contributions when $d=0$.

Proposition \ref{Prop_Stem_inv} with input $t=0$ gives:
\[
\begin{split}
&\sum_d Q^d\langle 1-P \rangle^{X_{\zeta}}_{0,1,d} \\
& =\sum_{n,d} \frac{Q^{rd}}{n!}  \Big[ 1-P , \leg_{\zeta}(L), \dots, \leg_{\zeta}(L), \delta_{\zeta}(L^{1/r}) \Big]^{X_{\zeta}}_{0,2+n,d}
\\
&= \sum_{d} Q^{rd} \Big[ 1-P,1-\zeta^{-1}L^{1/r} \Big]^{X_{\zeta}}_{0,2,d} + \Big[ 1-P, \leg_{\zeta}(L)  ,1-\zeta^{-1}L^{1/r} \Big]^{X_{\zeta}}_{0,3,0}
\\
&= \sum_d Q^{rd} \zeta^{-1} \Big( \frac{\GV^{(3)}_{d}}{r^2d^2} \Big)
+ \Big[  1-P, \leg_{\zeta}(L), 1-\zeta^{-1} L^{1/r} \Big]^{X_{\zeta}}_{0,3,0}
\\
&+ \Big[ 1-P, \Psi^r\Big( \frac{(1-P)^2}{5} \sum_k Q^k \langle 1-P \rangle^{\fake}_{0,1,k} \Big), 1-\zeta^{-1}L^{1/r} \Big]^{X_{\zeta}}_{0,3,0} 
\end{split}
\]
The last term corresponds to the contribution of the twisting of unramified locus in type $C$. 
It can be computed as follows:
\[
\begin{split}
    &\Big[ 1-P, \Psi^r\Big( \frac{(1-P)^2}{5} \sum_k Q^k \langle 1-P \rangle^{\fake}_{0,1,k} \Big), 1-\zeta^{-1}L^{1/r} \Big]^{X_{\zeta}}_{0,3,0} \\
    &= \Big[ 1-P, \frac{(1-P^r)^2}{5} \sum_k Q^{rk} (k\GW_k), 1-\zeta^{-1} \Big]^{X_{\zeta}}_{0,3,0}
    \\
    &= (1-\zeta^{-1}) \sum_k Q^{rk} (k\GW_k) = (1-\zeta^{-1}) \sum_k Q^{rk} \Big(\frac{1}{k^2} \GV^{(3)}_k\Big) . 
\end{split}
\]
Note that the constant $r^2$ from $(1-P^r)^2$ is cancelled by the constant in Lemma \ref{l:4.1} and the Poincar\'e pairing on $X/\mathbb{Z}_r$.

Assume the expression of $J^K$ in Theorem \ref{Thm_rpower} up to $Q^{M-1}$, a direct computation shows that:
\[
\begin{split}
\leg_{\zeta}(q) &= \Psi^r \Big(  [J^K_{q=1} ]_+ - (1-q) \Big) \\
&= \frac{(1-P^r)^2}{5} \sum_{d < M } Q^{rd} \Big(  \GV_d^{1} - \frac{1}{d^2} \GV^{(3)}_d \Big) 
\\
& + \frac{(1-P^r)^3}{5} \sum_{d<M}Q^{rd}\Big( d\GV_d^{(-1)} + \GV_d^{(1)} - \frac{ \GV_d^{(3)} }{d^2} - \frac{\GV_d^{(3)}}{d^3}  \Big)
\\
& \hspace{90mm} \mbox{(mod $Q^{M+1}$)}
\end{split}
\]
We recover the coefficient of $Q^M$ on $\leg_{\zeta}(q)$ since $r\geq 2$. Now we have
\[
\begin{split}
     & {\rm Coeff}\Big(\Big[  1-P, \leg_{\zeta}(L), 1-\zeta^{-1} L^{1/r} \Big]^{X_{\zeta}}_{0,3,0} ; Q^M \Big)
     \\
     & \hspace{60pt} =(1-\zeta^{-1}) \Big( \GV^{(1)}_{M/r} - \frac{r^2}{M^2}\GV^{(3)}_{M/r}  \Big).
\end{split}
\]
Compare the coefficient of $Q^M$ gives the result.
\end{proof}

\begin{lemma} \label{lemma_statement6}
    \[
    \langle 1 \rangle^{X_{\zeta}}_{0,1,M} =  (1-\zeta^{-1}) M\GV^{(-1)}_{M/r}  + \zeta^{-1} \Big(\frac{\GV^{(3)}_{M/r}}{M^3}\Big).
    \]
\end{lemma}
\begin{proof}
The computation is similar to Lemma \ref{lemma_statement5}. Proposition \ref{Prop_Stem_inv} with input $t=0$ gives:
\[
\begin{split}
&\sum_d Q^d \langle 1 \rangle^{X_{\zeta}}_{0,1,d}
\\
&= \sum_{n,d} \frac{Q^{rd}}{n!} \Big[ 1, \leg_{\zeta}(L),\dots, \leg_{\zeta}(L), \delta_{\zeta}(L^{1/r}) \Big]^{X_{\zeta}}_{0,2+n,d}
\\
& =\zeta^{-1} \Big( \frac{\GV^{(3)}_{d}}{r^3d^3} \Big) + \Big[  1, \leg_{\zeta}(L), 1-\zeta^{-1} L^{1/r} \Big]^{X_{\zeta}}_{0,3,0}
\\
& + \Big[ 1, \Psi^r\Big( \frac{(1-P)^3}{5} \sum_k Q^k \langle 1 \rangle^{\fake}_{0,1,k} \Big), 1-\zeta^{-1}L^{1/r} \Big]^{X_{\zeta}}_{0,3,0}.
\end{split}
\]
Note that 
\[
\begin{split}
    &\Big[ 1, \Psi^r\Big( \frac{(1-P)^3}{5} \sum_k Q^k \langle 1 \rangle^{\fake}_{0,1,k} \Big), 1-\zeta^{-1}L^{1/r} \Big]^{X_{\zeta}}_{0,3,0} \\
    &= \Big[ 1, \frac{(1-P^r)^3}{5} \sum_k Q^{rk} \GW_k, 1-\zeta^{-1} \Big]^{X_{\zeta}}_{0,3,0}
    \\
    &= r(1-\zeta^{-1}) \sum_k Q^{rk} (\GW_k) = r(1-\zeta^{-1}) \sum_k Q^{rk} \Big(\frac{1}{k^3} \GV^{(3)}_k\Big),
\end{split}
\]
and that
\[
\begin{split}
     & {\rm Coeff}\Big(\Big[  1, \leg_{\zeta}(L), 1-\zeta^{-1} L^{1/r} \Big]^{X_{\zeta}}_{0,3,0} ; Q^M \Big)
     \\
     & \hspace{60pt} =r(1-\zeta^{-1}) \Big( \frac{M \GV^{(-1)}_{M/r}}{r} - \frac{r^3}{M^3}\GV^{(3)}_{M/r}  \Big),
\end{split}
\]
by assuming the expression of $J^K$ in Theorem \ref{Thm_rpower} up to $Q^{M-1}$.

Compare the coefficient of $Q^M$ gives the result.
\end{proof}

\appendix

\section{Alternative proof of Theorem~\ref{Thm_deg0_inv}} \label{Appendix_degree0}
In this Appendix, we give another proof of Theorem~\ref{Thm_deg0_inv} using only reconstruction theorem.

We start with the localization at $q=1$, which will relate to the reconstruction theorem in cohomological theory, which is recalled as follows:
\[
\begin{split}
\frac{1}{-z} J^H (0) &= 1 + \sum_{d\geq 0} Q^d \Big( \frac{d\GW_d H^2}{(-z)^2} - \frac{2\GW_d H^3}{(-z)^3} \Big)
\\
&= \frac{1}{-z} \sum_{d\geq 0}I^H_d Q^d \exp \Big( \frac{(dz-H)\tau(Q)}{z} \Big) c(Q),
\end{split}
\]
for some uniquely determined $\tau(Q), c(Q) \in \mathbb{Q}[\![Q]\!]$.
\begin{theorem}
Degree $0$ terms with $\zeta=1$ are given as follows:
\[
\begin{split}
    a_{M,1} &=0,
    \\
    c_{M,1} &= \frac{\GV_M^{(3)}}{M^2} = M\GW_M ,
    \\
    f_{M,1} &= \frac{-2 \GV_M^{(3)}}{M^3} = -2\GW_M.
\end{split}
\]
\end{theorem}
\begin{proof}
We localize the reconstruction theorem at $q=1$:
\[
\begin{split}
& \frac{1}{1-q} J^K(0) \\
\equiv_{q=1}^0 & \sum_{d\geq 0} \Big(\frac{1}{1-q}I^K_d \Big)_{q=1} Q^d \exp \left( \sum_{k>0}\frac{ \sum_{i=1}^3 \Psi^k(\e_i(Q)) (1-P^kq^{kd})^i}{k(1-q^k)}  \right)_{q=1}  
\\
&\qquad\qquad \cdot \left(\sum_{i=0}^3 r_i(q,Q) (1-Pq^d)^i\right)_{q=1} .
\end{split}
\]
We analyze the $I^K$ term, the $\exp$ term and the polynomial term at the end individually.
\begin{lemma}
\[
\Big(  \frac{1}{1-q} I^{\KK} \Big)_{q=1} \equiv_{q=1}^0 \Big(\frac{1}{-z} I^H(0) \Big) \Big|_{z=-(1-q), H=1-P} 
\]
\end{lemma}
\begin{proof}
Since we consider only the degree $0$ part, it suffices to compute the following terms in the expansion of $I^{\KK}$:
\[
\begin{split}
    \left(\frac{ \prod_{r=1}^{5d} (1-q^r) }{( \prod_{r=1}^d (1-q^r) )^5} \right)_{q=1} &\equiv_{q=1}^0 \frac{(5d)!}{(d!)^5}
    \\
    \Big( e_i(k) (1-P)^i \Big)_{q=1} &\equiv_{q=1}^0 s_i(k) \Big( \frac{1-P}{1-q} \Big)  ^i,
\end{split}
\]
where 
\[
s_i(k) = \sum_{1 \leq j_1< \dots < j_i \leq k} \frac{1}{j_1} \dots \frac{1}{j_k}.
\]
After replacing all terms in the expansion of $I^{\KK}$ with their degree 0 part in the above computation and the substitution $(1-q) = -z$ and $1-P=H$, we get
\[
    \begin{split}
        &\frac{1}{-z}\sum_{d \geq 0} Q^d \frac{(5d)!}{(d!)^5} \Big( 1 + s_1(5d) (\frac{5H}{-z}) + e_2(5d) (\frac{5H}{-z})^2 + e_3(5d) (\frac{5H}{-z})^3 \Big)
        \\
        & \Big( 1 - 5 e_1(d) (\frac{H}{-z}) + (15e_1(d)^2 -5e_2(d)) (\frac{H}{-z})^2 - (35e_1(d)^3 -30e_1(d)e_2(d))(\frac{H}{-z})^3  \Big),
    \end{split}
\]
which is exactly the expansion of $\displaystyle \frac{1}{-z}I^H(0)$. This proves the lemma.
\end{proof}
\begin{lemma}
\[
\begin{split}
&\exp \left( \sum_{k>0}\frac{ \sum_{i=1}^3 \Psi^k(\e_i(Q)) (1-P^kq^{kd})^i}{k(1-q^k)}  \right)_{q=1}  \equiv_{q=1}^0
\\
&\qquad\exp \left(  \sum_{k > 0} \frac{ \Psi^k (T_1(Q)) (1-P) }{k(1-q)} \right)\exp \left(  \sum_{k>0} -\frac{d \Psi^k(T_1(Q))}{k} \right)
\end{split}
\]
\end{lemma}
\begin{proof}
For each $k$, we observe that
\[
\begin{split}
\frac{ 1-P^kq^{kd} }{k(1-q^k)} &\equiv_{q=1}^0 \frac{(1-P^k)}{k(1-q^k)} + \frac{1-q^{kd}}{k(1-q^k)}
 \equiv_{q=1}^0 \frac{(1-P)}{k(1-q)} + \frac{d}{k}
 \\
 \frac{ (1-P^kq^{kd})^i }{k(1-q^k)} &\equiv_{q=1}^0 0, \mbox{  for $i>1$}.
\end{split}
\]
The lemma follows.
\end{proof}
\begin{lemma}
\[
\sum_{i=0}^3 r_i(q,Q) (1-Pq^d)^i \equiv_{q=1}^0 r_0(1,Q)
\]
\end{lemma}
\begin{proof}
For the degree 0 term of the polynomial, we simply put $q=1$ and $P=1$. 
\end{proof}
Now we have 
\[
\begin{split}
    & \frac{1}{1-q}J^K(0) \\ 
    \equiv_{q=1}^0 & \left(\frac{1}{-z}\sum_{d\geq 0}I^H_d Q^d \exp \Big(-\tau(Q) \frac{H}{z}\Big) \exp \Big(d\tau(Q) \Big) \Big(c(Q) \Big)\right)\Big|_{z=-(1-q), H=1-P},
\end{split}
\]
where 
\[
\tau(Q) = - \sum_{k>0} \frac{\Psi^k (T_1(Q))}{k}, \quad c(Q) = r_0(1,Q).
\]
The above expression coincides with the  reconstruction theorem in cohomological theory:
\[
\begin{split}
\frac{1}{-z} J^H(0) &= \frac{1}{-z}\sum_{d\geq 0}I^H_d Q^d \exp \Big(-\tau(Q) \frac{H}{z}\Big) \exp \Big(d\tau(Q) \Big) \Big(c(Q) \Big)
\\
&= 1 + \sum_{d \geq 0} Q^d \Big(  \frac{d\GW_d H^2}{(-z)^2} - \frac{2\GW_d H^3}{(-z)^3} \Big).
\end{split}
\]
We conclude that the degree 0 term in the localization at $q=1$ coincides with cohomological computation after the substitution $1-q=-z$ and $1-P=H$.

\end{proof}

Let $\zeta$ be a primitive $r$-th roots of unity. We generalize the above computation to all degree 0 terms.
\begin{theorem} [=Theorem~\ref{Thm_deg0_inv}]
All degree 0 terms are given as follows:
\[
\begin{split}
    a_{M,\zeta} &=0,
    \\
    c_{M,\zeta} &= \frac{ \GV^{(3)}_{M/r}}{M^2} = \frac{\frac{M}{r} \GW_{M/r} }{r^2},
    \\
    f_{M,\zeta} &= \frac{-2 \GV^{(3)}_{M/r}}{M^3} = \frac{-2\GW_{M/r}}{r^3}.
\end{split}
\]
All invariants are defined to be zero if $M/r$ is not integer.
\end{theorem}

\begin{proof}
Using the reconstruction theorem again but localize at $q=\zeta^{-1}$, we have
\[
\begin{split}
    & \frac{1}{1-q} J^{\KK}(0)
    \\
    & \equiv_{q=\zeta^{-1}}^0 \Big[ \sum_{i\geq 0}  I^{\KK}_{ri}(q) Q^{ri} \exp \Big( \sum_{k>0} \frac{i \Psi^{rk} (T_1(Q))}{k} \Big) \Big] 
    \exp\Big( \sum_{k > 0}  \frac{ \Psi^{rk}(T_1(Q))(1-P) }{ rk(1-\zeta q) } \Big) 
    \\
    & \left[ \sum_{j=0}^{r-1} \pi_+( I_{j}^{\KK})(\zeta^{-1}) \exp\Big( \sum_{k>0} \frac{\Psi^k(T_1(Q))}{k} \sum_{l=0}^{j-1}q^{kl} \Big) \Big( \sum_{i=0}^3 r_i(\zeta^{-1},Q) (1-\zeta^{-j})^i \Big) Q^j      \right].
\end{split}
\]
For the above expression, we use the following observations:
\begin{itemize}
    \item For $m=ri+j$ with $0\leq j <r$, we have
    \[
    \begin{split}
        I^{\KK}_m(q) &\equiv_{q=\zeta^{-1}}^0 I^{\KK}_{rk}(q) \cdot \pi_+( I^{\KK}_j)(q) 
        \\
        & \equiv_{q=\zeta^{-1}}^0 I^{\KK}_{rk}(q) \cdot  \pi_+ (I^{\KK}_j) (\zeta^{-1}),
    \end{split}
    \]
    where $\pi_+: \K \rightarrow \K_+$ is the projection map.
    \item With the same notation, we have
    \[
    \begin{split}
        &\exp \Big( \sum_{k>0} \frac{ \Psi^k(T_1(Q)) (1-P^kq^{km}) }{k(1-q^k)} \Big)\\
        &\equiv_{q=\zeta^{-1}}^0 \exp \Big(  \sum_{k>0} \frac{ \Psi^k(T_1(Q)) (1-q^{km}) }{k(1-q^k)} \Big) \exp \Big(  \sum_{k>0} \frac{ \Psi^k(T_1(Q)) \cdot kq^{km}(1-P) }{k(1-q^k)} \Big) \\
        &\equiv_{q=\zeta^{-1}}^0 \exp \Big(  \sum_{k>0} \frac{ \Psi^k(T_1(Q))}{k} \sum_{l=0}^{m-1} q^{kl} \Big) \exp \Big(  \sum_{k>0} \frac{ \Psi^{rk}(T_1(Q)) (1-P) }{rk(1-\zeta q)} \Big) \\
        &\equiv_{q=\zeta^{-1}}^0 \exp \Big( \sum_{k>0} \frac{i \Psi^{rk} (T_1(Q))}{k} \Big)\exp \Big(  \sum_{k>0} \frac{ \Psi^k(T_1(Q))}{k} \sum_{l=0}^{j-1} q^{kl} \Big) 
        \\
        & \qquad\qquad\qquad\qquad\qquad\qquad\qquad\qquad \cdot \exp \Big(  \sum_{k>0} \frac{ \Psi^{rk}(T_1(Q)) (1-P) }{rk(1-\zeta q)} \Big).
    \end{split}
    \]
\end{itemize}
Our expression of $\displaystyle \frac{1}{1-q} J^K(0)$ fits into the following lemma:
\begin{lemma} \label{lemma_rpower}
Let $f,g \in \K$, and $h,h'\in K[q]$. Assume that
\[
f\cdot g = 1 \ ({\rm mod} Q), \quad f \cdot g \cdot h |_{\mathcal{K}_+} =1,
\]
and that
\[
\Psi^r ((f)_{q=1})(\zeta q ) \cdot \Psi^r ((g)_{q=1}) (\zeta q) \cdot h'(q) \Big|_{\mathcal{K}_+}  \equiv_{q= \zeta^{-1},0} 1.
\]
Then we have
\begin{enumerate}
    \item $h'(q) = \Psi^r (h)(\zeta q)$.
    \item 
    \[ 
    \begin{split}
    \Psi^r ((f)_{q=1})(\zeta q ) &\Psi^r ((g)_{q=1}) (\zeta q) h'(q)\Big|_{\mathcal{K}_-}
    \\
    &=  \Psi^r ((f)_{q=1})(\zeta q ) \Psi^r ((g)_{q=1}) (\zeta q) \Big|_{\mathcal{K}_-} \cdot \Psi^r(h(1)).
    \end{split}
    \]
\end{enumerate}
\end{lemma}
\begin{proof}
$h'$ is uniquely determined by induction on $Q$ since we assume that $h'\in K[q]$. It suffices to prove that 
\[
\Psi^r (h)(\zeta q) =  h'(q) = \Psi^r ( (h)_{q=1} )(\zeta q),
\]
which follows from $(h)_{q=1}(q) = h(q)$ since $h\in K[q]$. 

For the second statement, notice that
\[
\Psi^r (h)(\zeta q) \equiv_{q=\zeta^{-1}}^0 \Psi^r (h(1)) .
\]
\end{proof}
We claim that:
\[
\begin{split}
    &\sum_{j=0}^{r-1} \pi_+( I_{j}^K)(\zeta^{-1}) \exp\Big( \sum_{k>0} \frac{\Psi^k(T_1(Q))}{k} \sum_{l=0}^{j-1}q^{kl} \Big) \Big( \sum_{i=0}^3 r_i(\zeta^{-1},Q) (1-\zeta^{-j})^i \Big) Q^j \\
    = & \Psi^r (r_0(1,Q)). 
\end{split}
\]
We prove the claim by taking $h(q)= r_0(q,Q)$ in Lemma \ref{lemma_rpower}. To apply the lemma, we need the following two observations:
\begin{itemize}
    \item 
    \[
    \frac{1}{1-q} I^K(q) \equiv_{q=\zeta^{-1}}^0 \Psi^r \Big( ( \frac{1}{1-q} I^K(q)  )_{q=1} \Big)(\zeta q)_{1-P => \frac{1-P}{r}}.
    \]
    The description at the end means that replacing $1-P$ by $\frac{1-P}{r}$. The results follows from
    \[
    \begin{split}
    \frac{ \prod_{k=1}^{5rd} (1-q^k) }{ \prod_{k=1}^{rd} (1-q^k)^5 } \equiv_{q=\zeta^{-1}}^0  \frac{ (5d)! }{  (d!)^5 }, \quad \frac{1}{1-q^{rd}} \equiv_{q=\zeta^{-1}}^{0} \frac{1}{rd(1-\zeta q)}.
    \end{split}
    \]
    \item 
    \[
    \begin{split}
        & \exp\Big( \sum_{k > 0}  \frac{ \Psi^{rk}(T_1(Q))(1-P) }{ rk(1-\zeta_r q) } \Big) \\
        \equiv_{q=\zeta^{-1}}^0 & \Psi^r \Big( \exp \Big( \sum_{k>0} \frac{\Psi^k(T_1(Q))(1-P)}{k(1-q)} \Big) \Big)_{\substack{ q=>\zeta q \\ 1-P => \frac{1-P}{r}    }}.
    \end{split}
    \]
    The description at the end means that replacing $q$ by $\zeta q$ and $1-P$ by $\frac{1-P}{r}$.
\end{itemize}
Taking $\displaystyle f = \frac{1}{1-q}I^K(q) $ and $\displaystyle g =  \exp\Big( \sum_{k > 0}  \frac{ \Psi^{rk}(T_1(Q))(1-P) }{ rk(1-\zeta q) } \Big) $ in the lemma, we conclude the claim. Changing the variables from $1-P$ to $\frac{1-P}{r}$ does not change any arguments and conclusions in the lemma.

Finally, we get
\[
\begin{split}
\frac{1}{1-q} J^K(0) &\equiv_{q=\zeta^{-1}}^0 \Psi^r \Big( ( \frac{1}{1-q} J^K(0) )_{q=1} \Big)_{\substack{q=>\zeta q \\ 1-P => \frac{1-P}{r}   }}
\\
& = 1 + \frac{(1-P)^2}{r^2} \sum_{d\geq 1} Q^{rd} \Big(  \frac{d\GW_d}{(1-\zeta q)^2} \Big)
+ \frac{(1-P)^3}{r^3} \sum_{d\geq 1} Q^{rd} \Big( \frac{-2\GW_d}{(1-\zeta q)^3}  \Big).
\end{split}
\]
This proves the theorem.
\end{proof}

\bibliographystyle{alpha}

\bibliography{zbib}
    
\end{document}